\documentclass[a4paper,11pt,reqno]{article}
\usepackage{relsize}
\usepackage{color}
\usepackage{mathtools}
\usepackage[dvipsnames,table]{xcolor}
\usepackage[noadjust]{cite}

\usepackage{hyperref, enumitem}

\hypersetup{
  colorlinks   = true, 
  urlcolor     = blue, 
  linkcolor    = Purple, 
  citecolor   = red 
}
\usepackage{amsmath,amsthm,amssymb}
\usepackage{xurl}
\hypersetup{breaklinks=true}
\usepackage[margin=2.5cm]{geometry}

\usepackage{graphicx}
\usepackage{algorithm}
\usepackage{algorithmic}
\usepackage[T1]{fontenc}
\usepackage{authblk}
\usepackage[english]{babel}
\usepackage{stackrel}

\usepackage[capitalize,noabbrev]{cleveref}

\usepackage{tikz}
\usetikzlibrary{decorations.pathreplacing}
\usetikzlibrary{shapes.geometric,quotes}
\usetikzlibrary{shapes.misc}

\tikzset{cross/.style={cross out, draw=black, minimum size=2*(#1-\pgflinewidth), inner sep=0pt, outer sep=0pt}, 
cross/.default={1pt}}
\theoremstyle{definition}
\newtheorem{theorem}{Theorem}[section]

\newtheorem{proposition}[theorem]{Proposition}
\newtheorem{corollary}[theorem]{Corollary}
\newtheorem{lemma}[theorem]{Lemma}
\newtheorem{definition}[theorem]{Definition}
\newtheorem{example}[theorem]{Example}
\newtheorem{remark}[theorem]{Remark}

\newtheorem{conjecture}[theorem]{Conjecture}

\newcommand{\F}{\mathbb F}
\newcommand{\Z}{\mathbb Z}
\newcommand{\LL}{\mathbb L}
\newcommand{\Q}{\mathbb Q}

\newcommand{\mB}{\mathcal B}
\newcommand{\mC}{\mathcal C}
\newcommand{\C}{\mathcal C}
\newcommand{\mD}{\mathcal D}

\newcommand{\mF}{\mathcal F}
\newcommand{\mT}{\mathcal T}
\newcommand{\mS}{\mathcal S}
\newcommand{\mL}{\mathcal L}

\newcommand{\bsig}{\bar{\sigma}}

\newcommand{\gen}[1]{\langle #1\rangle}

\newcommand{\rev}[1]{{#1}}

\DeclareMathOperator{\rk}{rk}
\DeclareMathOperator{\id}{id}
\DeclareMathOperator{\Gal}{Gal}
\DeclareMathOperator{\End}{End}

\begin{document}

\title{A proof of the Etzion-Silberstein conjecture for monotone and MDS-constructible Ferrers diagrams}

\author[1,2]{Alessandro Neri}

\affil[1]{\begin{small}Department of
Mathematics and Applications
``R. Caccioppoli'',  University of Naples Federico II, Italy\end{small}}

\affil[2]{\begin{small}Department of Mathematics: Analysis, Logic and Discrete Mathematics, Ghent University, Belgium\end{small}}

\author[3]{Mima Stanojkovski}

\affil[3]{\begin{small}Department of Mathematics, University of Trento, Italy\end{small}}

\date{}

\maketitle

\begin{abstract}
    Ferrers diagram rank-metric codes were introduced by Etzion and Silberstein in 2009. In their work, they proposed a conjecture on the largest dimension of a space of matrices over a finite field whose nonzero elements are supported on a given Ferrers diagram and all have rank lower bounded by a fixed positive integer $d$. Since stated, the Etzion-Silberstein conjecture has been verified in a number of cases, often requiring additional constraints on the field size or on the minimum rank $d$ in dependence of the corresponding Ferrers diagram. As of today, this conjecture still remains widely open.
Using modular methods, we give a constructive proof of the Etzion-Silberstein conjecture for the class of strictly monotone Ferrers diagrams, which does not depend on the minimum rank $d$ and holds over every finite field. In addition, we leverage on the last result to also prove the conjecture for the class of MDS-constructible Ferrers diagrams, without requiring any restriction on the field size.
\newline
\newline
\emph{Keywords}: Rank-metric codes, Ferrers diagrams, Etzion-Silberstein conjecture, monotone diagrams, convex diagrams, MDS-constructible diagrams. 
\newline
2020 \emph{Mathematics Subject Classification}:  11T71, 15A03, 16S35, 94B05, 94B60.
\end{abstract}

\section{Introduction}

Let $\F$ be a finite field and let $1\leq d\leq n$ and $1\leq k\leq n^2$ be integers. Let, moreover, $\F^{n\times n}$ denote the collection of $n\times n$ matrices with coefficients in $\F$. Then the map 
\[
\F^{n\times n}\times \F^{n\times n}\longrightarrow\Z_{\geq 0}, \quad (A,B) \longmapsto \rk(A-B) 
\]
defines a metric on $\F^{n\times n}$. An $[n\times n,k,d]_{\F}$ \emph{rank-metric code} is a $k$-dimensional linear subspace $\mC$ of $\F^{n\times n}$ such that  $d$ is the \emph{minimum rank} among the nonzero elements of $\mC$. 
Rank-metric codes have a long mathematical history \cite{delsarte1978bilinear,gabidulin1985theory} and have become very popular in the field of algebraic coding theory since the discovery of their applications to network coding \cite{silvaetal2008a}.
Thanks to the work of Delsarte \cite{delsarte1978bilinear}, the parameters  of an $[n\times n, k ,d]_{\F}$ rank-metric code are related by the following Singleton-like Bound:
\begin{equation}\label{eq:SLBound}
k\leq n(n-d+1).
\end{equation}
Codes meeting the last bound with equality are called \emph{maximum rank distance codes} (MRDs, for short) and constructions of such codes are known over finite fields and can be found in \cite{delsarte1978bilinear,gabidulin1985theory}. 

If rank-metric codes allow for codes within the entire matrix space $\F^{n\times n}$, the more general \emph{Ferrers diagram rank-metric codes} are rank-metric codes with prescribed support. These were introduced in 2009 by Etzion and Silberstein \cite{etzion2009error} motivated by the application of subspace codes in network coding and arise from subspace codes in a fixed Schubert cell. The study of Ferrers diagram rank-metric codes has been carried on since by various authors \cite{antrobus2019maximal,ballico2015linear,etzion2016optimal,gorla2017subspace,gruica2022rook,liu2019constructions,liu2019several}.

Let $[n]=\{1,\ldots,n\}$ and let $\mD\subseteq [n]^2$
be a top-right justified \emph{Ferrers diagram} (cf.\ \cref{def:Ferrers}) outside of which the elements of the code will have entries equal to $0$.
A $[\mD,k,d]_{\F}$ \emph{Ferrers diagram rank-metric code}  is  an $[n\times n,k,d]_{\F}$ rank-metric code $\mC$ such that if $A=(a_{i,j})\in\mC$ and $(x,y)\in[n]^2\setminus\mD$, then $a_{x,y}=0$.
Two immediate examples are: 
\begin{enumerate}[label=$(\arabic*)$]
    \item If $\mD=[n]^2$, then a $[\mD,k,d]_{\F}$ Ferrers diagram rank-metric code  is the same as an $[n\times n,k,d]_{\F}$ rank-metric code. 
    \item If $\mD=\mT_n=\{(i,j): i,j\in[n] \textup{ with } i\leq j\}$, then a $[\mT_n,k,d]_{\F}$ Ferrers diagram rank-metric code is an $[n\times n,k,d]_{\F}$ rank-metric code that is contained in the space of upper triangular $n\times n$ matrices with coefficients over $\F$.
\end{enumerate}
Adding constraints to the support of the matrices in the code clearly influences the relation between its intrinsic parameters. In \cite{etzion2009error}, the upper bound on the dimension given in \eqref{eq:SLBound} is generalized to an explicit upper bound $\nu_{\min}(\mD,d)$ taking into account the prescribed support $\mD$; cf. \eqref{eq:gen_SBbound} for the exact value of $\nu_{\min}(\mD,d)$. For instance:
\begin{enumerate}[label=$(\arabic*)$]
    \item If $\mD=[n]^2$, then $\nu_{\min}(\mD,d)=n(n-d+1)$ recovers the bound in \eqref{eq:SLBound}.
    \item If $\mD=\mT_n$, then $\nu_{\min}(\mathcal{D},d)$ is equal to 
    \[
\nu_{\min}(\mT_n,d)=\sum_{i=1}^{n-d+1}i=\binom{n-d+2}{2}=\frac{(n-d+1)(n-d+2)}{2}.
\]
\end{enumerate}
A $[\mD,\nu_{\min}(\mD,d),d]_{\F}$  Ferrers diagram rank-metric code is called a  \emph{maximum Ferrers diagram code} (MFD, for short): Etzion and Silberstein conjectured that these codes exist for every choice of $d$ and $\mD$, and over every finite field $\F$. This is now known as the  \emph{Etzion-Silberstein conjecture} and, as of today, it is still open in its full generality.
In particular and to the best of our knowledge, not even the case of upper triangular matrices has been tackled completely. The instances of existing $[\mT_n,\nu_{\min}(\mT_n,d),d]_{\F}$ MFD codes that are known to us and that do not impose any restriction on the field $\F$ are the following: 
\begin{itemize}
    \item The case where $d=1$, which consists of all the upper triangular matrices. 
    \item The case where $d=2$, for which an MFD code is given by the subspace of upper triangular matrices where every diagonal sum is zero.
    \item The case where $d=3$, which is settled by \cite[Corollary~3.10]{antrobus2019maximal}. 
    \item The case where $d=n-1$, which is given by \cite[Theorem~5.3]{antrobus2019maximal} (for a constructive proof see \cite[Section~2.4]{antrobus2019phd}).
\end{itemize}
All other cases are settled with the additional assumption that $|\F|\geq n-1$, as shown in \cite{etzion2016optimal}. This is due to the fact that the pair $(\mT_n,d)$ is \emph{MDS-constructible} for every $d$ and $n$; cf. \cref{rem:triangular_MDS}.

\medskip

In this paper, we fill the gap by providing a constructive proof of the Etzion-Silberstein conjecture for the class of MDS-constructible Ferrers diagrams without imposing any condition on the field size; cf. \cref{thm:MDS-constr-column}.
We do this by first proving the conjecture for the subclass of \emph{strictly monotone Ferrers diagrams}; cf \cref{th:strictly-monotone}. These are Ferrers diagrams that do not allow two nonempty columns to have the same number of elements,  and include therefore the upper triangular Ferrers diagrams $\mT_n$ as a special subclass. By adjunction, we also prove the Etzion-Silberstein conjecture for \emph{initially convex} Ferrers diagrams; cf. \cref{def:strictmonotone_initiallyconvex}. 

We conclude this introduction with a short account of our methods. To produce MFD codes for strictly monotone Ferrers diagrams, we extend Delsarte-Gabidulin constructions of MRD codes, which are realized through the identification of $\F^{n\times n}$ with a skew algebra $\LL[\sigma]$, where $\LL$ is an auxiliary cyclic field extension of $\F$ of degree $n$ with Galois group generated by $\sigma$.
In \cref{thm:monotone_representation}, for a given finite field $\F$ of characteristic $p$ and $n=p^m$, we determine a skew algebra representation of the space of matrices supported on a Ferrers diagram $\mD\subseteq [n]^2$, when $\mD$ is \emph{$p$-monotone}; cf. \cref{def:p-monotone}. The crucial properties used are that $\LL$ has degree $p^m$ over $\F$ and that $\sigma-\id$ is a nilpotent $\F$-linear endomorphism of $\LL$. \rev{We subsequently employ \cref{thm:monotone_representation}} to prove the Etzion-Silberstein conjecture for $p$-monotone and $p$-convex Ferrers diagrams over finite fields of characteristic $p$ and where $n$ is a  power of $p$; cf. \cref{thm:MFD-monotone}.
For strictly monotone (and thus initially convex) Ferrers diagrams, we are able to move away from the modular case by showing that the property of being strictly monotone is preserved by embeddings.
Finally, we further extend our results to the class of MDS-constructible Ferrers diagrams, relying on the combinatorial features of their diagonals.

\bigskip
\noindent
\textbf{Organization of the paper.} In \cref{sec:notation} we set the notation for the whole paper, properly define Ferrers diagram codes and state the Etzion-Silberstein conjecture in its full generality. Additionally, we define MDS-constructible and monotone Ferrers diagrams as well as the class of $p$-monotone Ferrers diagrams and their adjoints. In \cref{sec:modular}, we concentrate on the case where $n=p^m$ is a power of the characteristic $p$ and, given a field extension $\LL$ of $\F$ of degree $n$, we study the symmetries of a natural full flag in $\LL$. We exploit the results of \cref{sec:modular} in \cref{sec:ferrers}, where we prove the Etzion-Silberstein conjecture for $p$-monotone Ferrers diagrams over finite fields of characteristic $p$. As corollaries we derive that the Etzion-Silberstein conjecture holds true over any finite field for the classes of strictly monotone, initially convex, and MDS-constructible Ferrers diagrams.

\bigskip
\noindent
\textbf{Acknowledgements.} 
Alessandro Neri is supported by the
Research Foundation--Flanders (FWO) grant 12ZZB23N.
Mima Stanojkovski is funded by the Italian program Rita Levi Montalcini, Edition 2020. This work was partially supported by Indam-GNSAGA (Italy).
They thank Ren\'e Schoof for a short proof of \cref{prop:absorbing} and the anonymous referees for their helpful comments.

\bigskip
\noindent
\textbf{Declaration of competing interest.}
The authors declare that they have no known competing financial interests or personal
relationships that could have appeared to influence the work reported in this paper.

\section{Notation and preliminaries}\label{sec:notation}

Let $n$ be a positive integer and let $\F$ be a field with a cyclic Galois extension $\LL$ of degree $n$. Fixing a generator $\sigma$ of the Galois group $\Gal(\LL/\F)$, the $\F$-algebra $(\End_\F(\LL),+,\circ)$ can be represented as a skew group algebra $(\LL[\sigma],+,\cdot)$. This is defined as the set
$$\LL[\sigma]=\left\{\sum_{i=0}^{n-1} a_i \sigma^i \,:\, a_i \in \LL\right\}$$
 endowed with usual addition, while the multiplication is defined on monomials via
$$(a\sigma^i)\cdot(b\sigma^j)=a\sigma^i(b)\sigma^{i+j}, \qquad \mbox{ for } a,b \in \LL,$$
and  extended by distributivity to $\LL[\sigma]$. The map 
$\LL[\sigma]\rightarrow\End_\F(\LL)$ given by 
\begin{equation}\label{eq:F-iso}
\begin{array}{ccc}
    \LL[\sigma] & \longrightarrow & \End_\F(\LL) \\
    \sum\limits_{i=0}^{n-1}a_i\sigma^i & \longmapsto & \Big(\alpha \longmapsto \sum\limits_{i=0}^{n-1}a_i\sigma^i(\alpha) \Big) 
\end{array}
\end{equation}
 is an $\F$-algebra isomorphism as a consequence of Artin's Theorem of linear independence of characters.
By choosing any $\F$-basis $\mB$ of $\LL$, one obtains an induced $\F$-algebra isomorphism 
\begin{equation}\label{eq:matrix_isomorphism}
\rev{\phi_{\mathcal{B}}: }\ \LL[\sigma] \ \ \longrightarrow\ \  \F^{n\times n},\end{equation}
\rev{where the columns of the matrix represent the images of the basis vectors, per usual Linear Algebra practice (in contrast with the common coding-theoretic costume of using row vectors)}. 
This representation lies at the heart of the most prominent constructions of families of optimal rank-metric codes, as we will see later in \cref{sec:rank-metric}.

Throughout, we identify every element $p(\sigma)\in \LL[\sigma]$ with its image in $\End_\F(\LL)$ under \eqref{eq:F-iso}: we denote its evaluation in $\alpha \in \LL$ by $p(\sigma)(\alpha)$. In addition, the \textbf{$\sigma$-degree} (or simply \textbf{degree}) of a nonzero element  $p(\sigma)=a_0\id+a_1\sigma+\ldots+a_{n-1}\sigma^{n-1}\in\LL[\sigma]$ is defined in the usual way as 
$$\deg p(\sigma)=\max\{ i\in\{0,\ldots,n-1\} \,:\,a_i\neq 0\}.$$ 
Artin's theorem of linear independence of characters ensures then that, for any nonzero $\sigma$-polynomial $p(\sigma)=a_0\id+a_1\sigma+\ldots+a_{n-1}\sigma^{n-1}\in\LL[\sigma]$, the following holds:
\begin{equation}\label{eq:bound_degree}
 \dim_\F(\ker p(\sigma))\leq \deg p(\sigma).
\end{equation}
This last statement can be found for instance in \cite[Theorem 5]{gow2009galois}, although it can also be derived from earlier works; see e.g.\ \cite{delsarte1978bilinear,guralnick1994invertible}.

\subsection{Ferrers diagram rank-metric codes}\label{sec:rank-metric}

Rank-metric codes were first introduced by Delsarte in \cite{delsarte1978bilinear} and then reintroduced by Gabidulin in \cite{gabidulin1985theory}. They have gained in popularity thanks to their application to network coding \cite{silvaetal2008a} and for further applications in \rev{distributed storage \cite{blaum2013partial} and cryptography \cite{gabidulin1991ideals}.}

\begin{definition}
Let $\F$ be a field and $k,m,n,d$ positive integers. An $[n\times m,k,d]_\F$ \textbf{rank-metric code}  is a $k$-dimensional $\F$-subspace $\C$ of $\F^{n\times m}$, where $d=\min\{\rk(A) : A\in \C\setminus \{0\}\}$ is called the \textbf{minimum rank distance} of $\C$. 
\end{definition}

The parameters $k,m,n,d$ of a rank-metric code depend from each other through the well-known Singleton-like bound, originally proved by Delsarte in \cite{delsarte1978bilinear} and reading:
\begin{equation}\label{eq:SingBound}
    k\leq \min\{m(n-d+1), n(m-d+1)\}.
\end{equation}
Codes meeting the bound of \eqref{eq:SingBound} with equality are called \textbf{maximum rank distance codes}, or simply \textbf{MRD codes}. When $\F$ admits a degree $\max\{n,m\}$ cyclic extension field $\LL$, there is a well-known construction of MRD codes for every choice of $d,n,m$ with $1\leq d \leq \min\{m,n\}$, which is due to Delsarte \cite{delsarte1978bilinear} for finite fields and to Guralnick \cite{guralnick1994invertible} for general fields. Given $\sigma$ a generator of $\Gal(\LL/\F)$, this construction goes as follows.

If $n=m$, then the $\LL$-subspace 
$$\LL[\sigma]_{n-d}=\left\{\sum_{i=0}^{n-d} a_i \sigma^i\,:\, a_i \in \LL \right\}\subseteq \LL[\sigma].$$
has $\F$-dimension $n(n-d+1)$ and can be seen as an $[n\times n,n(n-d+1),d]_{\F}$ code, due to the isomorphism in \eqref{eq:matrix_isomorphism}. Furthermore, by \eqref{eq:bound_degree} it is clear that the minimum rank distance of $\C$ is $d$, implying that $\LL[\sigma]_{n-d}$ is an MRD code. 

If $m<n$, then by the last case we know that $\LL[\sigma]_{m-d}$ is an $[n\times n, n(m-d+1),d+n-m]_\F$ MRD code. Projecting $\LL[\sigma]_{m-d}$ to $\F^{n\times m}$ by erasing any set of $n-m$ columns, we then obtain an $[n\times m,n(m-d+1),d]_\F$ MRD code. 

If $m>n$, one reduces to the previous case by interchanging the values $n$ and $m$ and transposing the matrices. 

\medskip 

In 2009, motivated by the application of subspace codes to network coding, Etzion and Silberstein initiated the study of Ferrers diagram rank-metric codes \cite{etzion2009error}. These are rank-metric codes whose elements have entries which are identically zero outside of a given Ferrers diagram, and arise from subspace codes in a fixed Schubert cell.

  \begin{definition}\label{def:Ferrers}
     A (top-right justified) \textbf{Ferrers diagram of order $n$}  is 
     %a  right-justified subset of $[n]^2$. In other words, $\mD\subseteq [n]^2$
      a subset $\mD$ of $[n]^2$
     satisfying the following properties:
     \begin{itemize}
     \item If $(i,j)\in \mD$ and $j' \in \{j,\ldots,n\}$, then $(i,j')\in \mD$.
     \item If $(i,j)\in\mD$ and $i'\in\{1,\ldots,i\}$, then $(i',j)\in\mD$.
     \end{itemize}
    \end{definition}
    
    Every Ferrers diagram $\mD$ can be graphically represented as a top-right justified grid of dots. We will often use this representation in order to help the reader graphically catch the features of the Ferrers diagrams that we will deal with in this paper. For compactness, however, we will also identify a Ferrers diagram $\mD$ with the vector $(c_1,\ldots,c_n)$, where $c_i\in\{0,\ldots,n\}$ is the number of dots  in the column $i$ of the grid: in view of \cref{def:Ferrers} it then holds that $0\leq c_1\leq\ldots\leq c_n\leq n$. 

    \begin{remark}
We have already seen that, if a Ferrers diagram is given as a subset of $[n]^2$, then its vector representation $(c_1,\ldots,c_n)$ is obtained as
\[
c_j= |\{i\in [n]\, :\,(i,j)\in \mD \}|.
\]
Conversely, if $\mD$ is given by a vector $(c_1,\ldots,c_n)$, then its representation as a subset of $[n]^2$ is derived as
\[
\mD=\{(i,j)\in[n]^2\, :\, i\leq c_j\}.
\]
These maps are inverses to each other and we will use the two representations of a Ferrers diagram interchangeably. 
    \end{remark}

   \begin{example}\label{ex:fer5}
        The Ferrers diagram of order $5$ given by $\mD=(0,1,1,4,5)$ is graphically represented by the following dotted grid. 
        
  \bigskip 
            
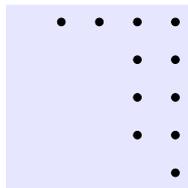
\centering
                    \begin{tikzpicture}[scale=0.5]
\draw[help lines, very thick, white, fill=blue!10] (0,1) -- (0,-4) -- (5,-4)--(5,1)--(0,1);
\draw[fill=black] (1.5,0.5) circle (0.1cm);
\draw[fill=black] (2.5,0.5) circle (0.1cm);
\draw[fill=black] (3.5,0.5) circle (0.1cm);
\draw[fill=black] (4.5,0.5) circle (0.1cm);
\draw[fill=black] (3.5,-0.5) circle (0.1cm);
\draw[fill=black] (4.5,-0.5) circle (0.1cm);
\draw[fill=black] (3.5,-1.5) circle (0.1cm);
\draw[fill=black] (4.5,-1.5) circle (0.1cm);
\draw[fill=black] (3.5,-2.5) circle (0.1cm);
\draw[fill=black] (4.5,-2.5) circle (0.1cm);
\draw[fill=black] (4.5,-3.5) circle (0.1cm);
\end{tikzpicture}
    \end{example}

    There is a natural notion of adjunction for Ferrers diagrams, which corresponds to the conjugate partition of an integer. 

    \begin{definition}
        Let $\mD$ be a Ferrers diagram of order $n$. Then, the \textbf{adjoint Ferrers diagram}   $\mD^\top$ is defined as 
        $$ \mD^\top=\{(n+1-j,n+1-i) \,:\, (i,j) \in \mD \}.$$
       
    \end{definition}
 In other words, the adjoint Ferrers diagram is obtained by transposing the indices of the original Ferrers diagram with respect to the main antidiagonal. We remark that, if $\mD$ is a Ferrers diagram, then $\mD^\top$ is the Ferrers diagram obtained from the reflection, with respect to the main antidiagonal, of a the graphical representation of $\mD$. This is illustrated in the following example.
       \begin{example}\label{ex:fer5^perp}
        The adjoint $\mD^\top$ of the Ferrers diagram $\mD=(0,1,1,4,5)$ from \cref{ex:fer5} is $\mD^\top=(1,2,2,2,4)$ and is graphically represented by the following. 
        
  \bigskip 
            
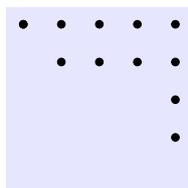
\centering
                    \begin{tikzpicture}[scale=0.5]
\draw[help lines, very thick, white, fill=blue!10] (0,1) -- (0,-4) -- (5,-4)--(5,1)--(0,1);
\draw[fill=black] (0.5,0.5) circle (0.1cm);
\draw[fill=black] (1.5,0.5) circle (0.1cm);
\draw[fill=black] (2.5,0.5) circle (0.1cm);
\draw[fill=black] (3.5,0.5) circle (0.1cm);
\draw[fill=black] (4.5,0.5) circle (0.1cm);
\draw[fill=black] (0.5,0.5) circle (0.1cm);
\draw[fill=black] (1.5,-0.5) circle (0.1cm);
\draw[fill=black] (2.5,-0.5) circle (0.1cm);
\draw[fill=black] (3.5,-0.5) circle (0.1cm);
\draw[fill=black] (4.5,-0.5) circle (0.1cm);
\draw[fill=black] (4.5,-1.5) circle (0.1cm);
\draw[fill=black] (4.5,-2.5) circle (0.1cm);
\end{tikzpicture}
    \end{example}

    For a given Ferrers diagram $\mD=(c_1,\ldots,c_n)$ of order $n$, denote by $\F^{\mD}$ the space of $n\times n$ matrices over $\F$ supported on $\mD$, that is,
    $$ \F^{\mD}=\{A=(a_{i,j})\in\F^{n\times n} \,:\, i>c_j \textup{ implies } a_{i,j}=0\}\subseteq \F^{n\times n}.$$

    \begin{example}
        If $\mD=(1,2,3,\ldots,n)$, then $\F^\mD$ is nothing else than the collection of upper triangular $n\times n$ matrices over $\F$. In this case, $\mD$ satisfies $\mD=\mD^\top$. 
    \end{example}

   \begin{definition}
       Let $\mD$ be a Ferrers diagram of order $n$. A $[\mD,k,d]_{\F}$ \textbf{Ferrers diagram rank-metric code} is a $k$-dimensional $\F$-subspace $\mC$ of $\F^\mD$ endowed with the rank distance. The parameter $d$ is the \textbf{minimum rank distance} of $\mC$ and is defined as
       $$d=\min\{\rk (A) \,:\, A \in \mC, A\neq 0\}.$$
   \end{definition}
    In their seminal paper \cite{etzion2009error}, Etzion and Silberstein showed the following relation between the parameters of a $[\mD,k,d]_{\F}$ code $\mC$, relying on the classical argument that $\mC$ trivially intersects any subspace of $\F^\mD$ whose elements have rank at most $d-1$. 

    \begin{proposition}[{\cite[Theorem 1]{etzion2009error}}]\label{prop:bound_Ferrers} 
     Let $\mD=(c_1,\ldots,c_n)$ be a Ferrers diagram of order $n$, and let $\mC$ be a $[\mD,k,d]_{\F}$ code. Then
     $$ k\leq\rev{ \min \left\{\sum_{i=1}^{n-j}\max\{0,c_i-d+1+j\}\ :\ j\in\{0,\ldots,d-1\}\right\}}. $$
    \end{proposition}

    The quantity on the right-hand side of the bound in Proposition \ref{prop:bound_Ferrers}, which we denote
    \begin{equation}\label{eq:gen_SBbound}\nu_{\min}(\mD,d)=\rev{\min \left\{\sum_{i=1}^{n-j}\max\{0,c_i-d+1+j\}\ :\ j\in\{0,\ldots,d-1\}\right\}},
    \end{equation}
    represents, informally speaking, the minimum, among all the $j\in\{0,\ldots,d-1\}$ of the number of dots remaining in the Ferrers diagram $\mD$ after removing the first $d-j-1$ rows and the last $j$ columns from $\mD$. From this interpretation, it is immediately clear that $\nu_{\min}(\mD,d)=\nu_{\min}(\mD^\top,d)$. 
    
    \rev{
    To lighten the notation, in the remaining part of the paper, we will write $\nu_j(\mD,d)$ meaning
    \[
    \nu_j(\mD,d)=\sum_{i=1}^{n-j}\max\{0,c_i-d+1+j\}
    \]
    yielding in turn that $\nu_{\min}(\mD,d)=\min\{\nu_j(\mD,d)\ :\ j\in\{0,\ldots,d-1\}\}$.
    }

    \begin{example}
Let $\mD=(0,1,1,4,5)$ be as in \cref{ex:fer5} and set $d=3$.     Then 
    \[
    \nu_{\min}(\mD,d)=\min\left\{\sum_{i=1}^{5-j}\max\{0,c_i-2+j\}\ :\ j\in\{0,1,2\}\right\}=\min\{2+3, 3, 1+1\}=2,
    \]
    which is the same as the minimum among the number of dots remaining after perfoming the following deletions:

      \bigskip 
            
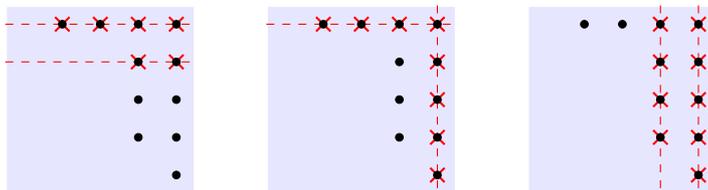
\centering
                    \begin{tikzpicture}[scale=0.5]
\draw[help lines, very thick, white, fill=blue!10] (0,1) -- (0,-4) -- (5,-4)--(5,1)--(0,1);
% \draw[red, very thick] (2,1)-|(2,-2)-|(5,-2)-| (5,1) -|cycle;
\draw[red, dashed] (0,0.5) -- (5,0.5);
\draw[red, dashed] (0,-0.5) -- (5,-0.5);
\draw (1.5,0.5) node [cross=3.5pt,red,thick] {};
\draw (2.5,0.5) node [cross=3.5pt,red,thick] {};
\draw (3.5,0.5) node [cross=3.5pt,red,thick] {};
\draw (4.5,0.5) node [cross=3.5pt,red,thick] {};
\draw (3.5,-0.5) node [cross=3.5pt,red,thick] {};
\draw (4.5,-0.5) node [cross=3.5pt,red,thick] {};
\draw[fill=black] (1.5,0.5) circle (0.1cm);
\draw[fill=black] (2.5,0.5) circle (0.1cm);
\draw[fill=black] (3.5,0.5) circle (0.1cm);
\draw[fill=black] (4.5,0.5) circle (0.1cm);
\draw[fill=black] (3.5,-0.5) circle (0.1cm);
\draw[fill=black] (4.5,-0.5) circle (0.1cm);
\draw[fill=black] (3.5,-1.5) circle (0.1cm);
\draw[fill=black] (4.5,-1.5) circle (0.1cm);
\draw[fill=black] (3.5,-2.5) circle (0.1cm);
\draw[fill=black] (4.5,-2.5) circle (0.1cm);
\draw[fill=black] (4.5,-3.5) circle (0.1cm);
\end{tikzpicture} \quad \quad
                   \begin{tikzpicture}[scale=0.5]
\draw[help lines, very thick, white, fill=blue!10] (0,1) -- (0,-4) -- (5,-4)--(5,1)--(0,1);
\draw[red, dashed] (0,0.5) -- (5,0.5);
\draw[red, dashed] (4.5,1) -- (4.5,-4);
\draw (1.5,0.5) node [cross=3.5pt,red,thick] {};
\draw (2.5,0.5) node [cross=3.5pt,red,thick] {};
\draw (3.5,0.5) node [cross=3.5pt,red,thick] {};
\draw (4.5,0.5) node [cross=3.5pt,red,thick] {};
\draw (4.5,-0.5) node [cross=3.5pt,red,thick] {};
\draw (4.5,-1.5) node [cross=3.5pt,red,thick] {};
\draw (4.5,-2.5) node [cross=3.5pt,red,thick] {};
\draw (4.5,-3.5) node [cross=3.5pt,red,thick] {};

\draw[fill=black] (1.5,0.5) circle (0.1cm);
\draw[fill=black] (2.5,0.5) circle (0.1cm);
\draw[fill=black] (3.5,0.5) circle (0.1cm);
\draw[fill=black] (4.5,0.5) circle (0.1cm);
\draw[fill=black] (3.5,-0.5) circle (0.1cm);
\draw[fill=black] (4.5,-0.5) circle (0.1cm);
\draw[fill=black] (3.5,-1.5) circle (0.1cm);
\draw[fill=black] (4.5,-1.5) circle (0.1cm);
\draw[fill=black] (3.5,-2.5) circle (0.1cm);
\draw[fill=black] (4.5,-2.5) circle (0.1cm);
\draw[fill=black] (4.5,-3.5) circle (0.1cm);
\end{tikzpicture} \quad \quad
                   \begin{tikzpicture}[scale=0.5]
\draw[help lines, very thick, white, fill=blue!10] (0,1) -- (0,-4) -- (5,-4)--(5,1)--(0,1);
\draw[red, dashed] (3.5,1) -- (3.5,-4);
\draw[red, dashed] (4.5,1) -- (4.5,-4);
\draw (3.5,0.5) node [cross=3.5pt,red,thick] {};
\draw (3.5,-0.5) node [cross=3.5pt,red,thick] {};
\draw (3.5,-1.5) node [cross=3.5pt,red,thick] {};
\draw (3.5,-2.5) node [cross=3.5pt,red,thick] {};
\draw (4.5,0.5) node [cross=3.5pt,red,thick] {};
\draw (4.5,-0.5) node [cross=3.5pt,red,thick] {};
\draw (4.5,-1.5) node [cross=3.5pt,red,thick] {};
\draw (4.5,-2.5) node [cross=3.5pt,red,thick] {};
\draw (4.5,-3.5) node [cross=3.5pt,red,thick] {};
\draw[fill=black] (1.5,0.5) circle (0.1cm);
\draw[fill=black] (2.5,0.5) circle (0.1cm);
\draw[fill=black] (3.5,0.5) circle (0.1cm);
\draw[fill=black] (4.5,0.5) circle (0.1cm);
\draw[fill=black] (3.5,-0.5) circle (0.1cm);
\draw[fill=black] (4.5,-0.5) circle (0.1cm);
\draw[fill=black] (3.5,-1.5) circle (0.1cm);
\draw[fill=black] (4.5,-1.5) circle (0.1cm);
\draw[fill=black] (3.5,-2.5) circle (0.1cm);
\draw[fill=black] (4.5,-2.5) circle (0.1cm);
\draw[fill=black] (4.5,-3.5) circle (0.1cm);
\end{tikzpicture}

    \end{example}
    
    In the same paper \cite{etzion2009error}, the authors also conjectured that over \textbf{finite} fields the bound given in Proposition \ref{prop:bound_Ferrers} is tight for any Ferrers diagram.  This is now known as the \textbf{Etzion-Silberstein conjecture} and is stated below.

    \begin{conjecture}[{\cite[Conjecture 1]{etzion2009error}}]\label{conj:ES}
        Let $\mD=(c_1,\ldots,c_n)$ be a Ferrers diagram of order $n$, and let $d \in \{1,\ldots,n\}$. If $\F$ is a finite field, then there exists a $[\mD,k,d]_{\F}$ code $\mC$ with 
        $$k=\nu_{\min}(\mD,d).$$
    \end{conjecture}

    A code meeting the bound from Proposition \ref{prop:bound_Ferrers} with equality is said to be a \textbf{maximum Ferrers diagram (MFD) code}. Thus, the Etzion-Silberstein conjecture (Conjecture \ref{conj:ES}) states that MFD codes exist over any finite field and for every Ferrers diagram of order $n$ and minimum rank distance $1\le d\le n$. Since $d=1$ is a trivial instance of the problem, we will often directly consider $d\geq 2$.

    \medskip

Though Conjecture \ref{conj:ES} is still widely open, there is a number of cases in which it has been confirmed; see  \cite[Section 2.3]{liu2019constructions} for a summary. Two main techniques have been used for this: the construction of MFD codes as subcodes of MRD codes and the concept of MDS-constructibility.  

The first one allows to completely settle the case where $d=2$: see for instance \cite[Theorem~2]{etzion2009error}, \cite[Corollary~19]{gorla2017subspace}), and \cite[Theorem~III.1]{antrobus2019maximal}. 
 For $d\ge 3$, the structure of MRD codes has been exploited in various ways for the construction of MFD codes, often with additional restrictions on the minimum rank $d$ depending on the Ferrers diagram $\mD$. This method can already be found in Etzion and Silberstein's seminal paper \cite{etzion2009error}, and it has been improved and refined ever since (see e.g.\ \cite{antrobus2019maximal,etzion2016optimal,gorla2017subspace,liu2019constructions,liu2019several}). A summary of the MRD-subcodes constructions can be found in \cite[Theorem 2.8]{liu2019several}. 

The second general method uses the theory of MDS codes in the Hamming metric, but only allows to prove Conjecture \ref{conj:ES} over ``large enough'' finite fields and for  Ferrers diagrams with the property of being \emph{MDS-constructible}; cf.\ \cref{def:MDS-constructible}. In the following, for $i\in[n]$, let  
$$\Delta_i^n=\{(j,j+i-1) : j\in[n-i+1]\}$$ 
be the $i$-th diagonal of an $n\times n$ grid. 

\begin{definition}\label{def:MDS-constructible}
 Let $\mD\subseteq [n]^2$ be a Ferrers diagram of order $n$, and let $d\in[n]$ be a positive integer. The pair $(\mD,d)$ is \textbf{MDS-constructible} if 
 $$ \nu_{\min}(\mD,d)=\sum_{i=1}^n\max\{0,|\mD\cap \Delta_i^n|-d+1\}.$$
\end{definition}

We remark that, since $|\Delta_i^n|=n-i+1$, for an MDS-constructible pair $(\mD,d)$ one  has
\begin{equation}\label{eq:MDS-smallsum}
    \nu_{\min}(\mD,d)=\sum_{i=1}^{n-d+1}\max\{0,|\mD\cap \Delta_i^n|-d+1\}.
\end{equation}
The notion   of MDS-constructible pairs was introduced in \cite{antrobus2019maximal}, but it was already shown in \cite[Theorem 7]{etzion2016optimal} that $[\mD,\nu_{\min}(\mD,d),d]_{\F}$ MFD codes exists (and can be explicitly constructed) whenever $|\F|\geq \max \{|\mD\cap \Delta_i^n|-1 \,:\, i\in[n]\}$. The rough idea is to embed MDS codes of length $|\mD\cap \Delta_i^n|$ and Hamming distance $d$ into $\mD\cap \Delta_i^n$ and then take the direct sum of these spaces. For a detailed description of this construction we refer the reader to  \cite[Construction 1]{etzion2016optimal}.  
We remark that a similar diagonal construction for rank metric codes was already given in \cite[Theorem 3]{roth1991maximum}.
Recently, existence results for MFD codes over some MDS-constructible pairs were derived in \cite{gruica2022rook}, with partial improvements on the field size.

\begin{remark}\label{rem:triangular_MDS} It is readily seen that the upper triangular Ferrers diagrams fall into this category. In other words, for positive integers $n$ and $1\leq d\leq n$, if $\mT_n=\{(i,j) \,:\, i,j\in[n], i \le j\}$ denotes the upper-triangular $n\times n$ Ferrers diagram, then
the pair $(\mT_n,d)$ is MDS-constructible.
\end{remark}

In \cref{sec:ferrers} we prove \cref{conj:ES} for the family of MDS-constructible Ferrers diagrams over any finite field. To get rid of the
restriction on the field size, we exploit 
 MRD-subcode constructions. The novelty of our approach is the use of an auxiliary extension field whose degree is a power of the characteristic.

\subsection{Monotone and convex Ferrers diagrams}\label{sec:monotone}

Two special families of Ferrers diagrams which we will focus on, are those of $p$-monotone and $p$-convex Ferrers diagrams, where $p$ is a prime number. We define them both in this section. 

   \begin{definition}
        A Ferrers diagram $\mD=(c_1,\ldots,c_n)$ of order $n$ is called \textbf{monotone} if, whenever $i\in\{1,\ldots,n-1\}$  is such that $0<c_i<n$, one has $c_{i+1} >c_i$.
    \end{definition}

        \begin{example}\label{exa:monotone}
    The following are examples of monotone Ferrers diagrams of order $n=5$.
    $$ \mD=(0,0,1,3,4), \;\; \mD=(1,2,4,5,5),\;\; \mD=(2,3,5,5,5),\;\; \mD=(0,1,4,5,5)$$
   % \bigskip 
    
                    
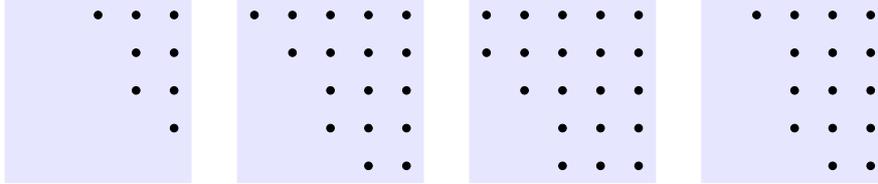
\centering
                           \begin{tikzpicture}[scale=0.5]
\draw[help lines, very thick, white, fill=blue!10] (0,1) -- (0,-4) -- (5,-4)--(5,1)--(0,1);
\draw[fill=black] (2.5,0.5) circle (0.1cm);
\draw[fill=black] (3.5,0.5) circle (0.1cm);
\draw[fill=black] (4.5,0.5) circle (0.1cm);
\draw[fill=black] (3.5,-0.5) circle (0.1cm);
\draw[fill=black] (4.5,-0.5) circle (0.1cm);
\draw[fill=black] (3.5,-1.5) circle (0.1cm);
\draw[fill=black] (4.5,-1.5) circle (0.1cm);
\draw[fill=black] (4.5,-2.5) circle (0.1cm);
\end{tikzpicture}  \quad
                    \begin{tikzpicture}[scale=0.5]
\draw[help lines, very thick, white, fill=blue!10] (0,1) -- (0,-4) -- (5,-4)--(5,1)--(0,1);
\draw[fill=black] (0.5,0.5) circle (0.1cm);
\draw[fill=black] (1.5,0.5) circle (0.1cm);
\draw[fill=black] (2.5,0.5) circle (0.1cm);
\draw[fill=black] (3.5,0.5) circle (0.1cm);
\draw[fill=black] (4.5,0.5) circle (0.1cm);
\draw[fill=black] (1.5,-0.5) circle (0.1cm);
\draw[fill=black] (2.5,-0.5) circle (0.1cm);
\draw[fill=black] (3.5,-0.5) circle (0.1cm);
\draw[fill=black] (4.5,-0.5) circle (0.1cm);
\draw[fill=black] (2.5,-1.5) circle (0.1cm);
\draw[fill=black] (3.5,-1.5) circle (0.1cm);
\draw[fill=black] (4.5,-1.5) circle (0.1cm);
\draw[fill=black] (2.5,-2.5) circle (0.1cm);
\draw[fill=black] (3.5,-2.5) circle (0.1cm);
\draw[fill=black] (4.5,-2.5) circle (0.1cm);
\draw[fill=black] (3.5,-3.5) circle (0.1cm);
\draw[fill=black] (4.5,-3.5) circle (0.1cm);
\end{tikzpicture}         \quad           \begin{tikzpicture}[scale=0.5]
\draw[help lines, very thick, white, fill=blue!10] (0,1) -- (0,-4) -- (5,-4)--(5,1)--(0,1);
\draw[fill=black] (0.5,0.5) circle (0.1cm);
\draw[fill=black] (1.5,0.5) circle (0.1cm);
\draw[fill=black] (2.5,0.5) circle (0.1cm);
\draw[fill=black] (3.5,0.5) circle (0.1cm);
\draw[fill=black] (4.5,0.5) circle (0.1cm);
\draw[fill=black] (0.5,-0.5) circle (0.1cm);
\draw[fill=black] (1.5,-0.5) circle (0.1cm);
\draw[fill=black] (2.5,-0.5) circle (0.1cm);
\draw[fill=black] (3.5,-0.5) circle (0.1cm);
\draw[fill=black] (4.5,-0.5) circle (0.1cm);
\draw[fill=black] (1.5,-1.5) circle (0.1cm);
\draw[fill=black] (2.5,-1.5) circle (0.1cm);
\draw[fill=black] (3.5,-1.5) circle (0.1cm);
\draw[fill=black] (4.5,-1.5) circle (0.1cm);
\draw[fill=black] (2.5,-2.5) circle (0.1cm);
\draw[fill=black] (3.5,-2.5) circle (0.1cm);
\draw[fill=black] (4.5,-2.5) circle (0.1cm);
\draw[fill=black] (2.5,-3.5) circle (0.1cm);
\draw[fill=black] (3.5,-3.5) circle (0.1cm);
\draw[fill=black] (4.5,-3.5) circle (0.1cm);
\end{tikzpicture}               \quad           \begin{tikzpicture}[scale=0.5]
\draw[help lines, very thick, white, fill=blue!10] (0,1) -- (0,-4) -- (5,-4)--(5,1)--(0,1);
\draw[fill=black] (1.5,0.5) circle (0.1cm);
\draw[fill=black] (2.5,0.5) circle (0.1cm);
\draw[fill=black] (3.5,0.5) circle (0.1cm);
\draw[fill=black] (4.5,0.5) circle (0.1cm);
\draw[fill=black] (2.5,-0.5) circle (0.1cm);
\draw[fill=black] (3.5,-0.5) circle (0.1cm);
\draw[fill=black] (4.5,-0.5) circle (0.1cm);
\draw[fill=black] (2.5,-1.5) circle (0.1cm);
\draw[fill=black] (3.5,-1.5) circle (0.1cm);
\draw[fill=black] (4.5,-1.5) circle (0.1cm);
\draw[fill=black] (2.5,-2.5) circle (0.1cm);
\draw[fill=black] (3.5,-2.5) circle (0.1cm);
\draw[fill=black] (4.5,-2.5) circle (0.1cm);
\draw[fill=black] (3.5,-3.5) circle (0.1cm);
\draw[fill=black] (4.5,-3.5) circle (0.1cm);
\end{tikzpicture}   
\end{example}

    \begin{remark}\label{rem:monotonicity}
        Let  $\mD=(c_1,\ldots,c_n)$ be a monotone Ferrers diagram of order $n$, and define $$\ell=\min\{i \,:\, c_i \neq 0\}\ \ \textup{ and }\ \  r=\max\{ i \,:\, c_i\neq n\}.$$ Then, for every $i,j \in \{\ell,\ell+1,\ldots,r+1\}$ with $i<j$, one has $c_j-c_i\geq j-i$.
    \end{remark}

    \begin{remark}\label{rem:monotone_diagonal}
        Let $\mD=(c_1,\ldots,c_n)$ be a monotone Ferrers diagram of order $n$. If $i,j<n$ are such that  $(i,j)\in \mD$, then monotonicity yields $(i+1,j+1)\in \mD$.
    \end{remark}

    \begin{definition}
        A Ferrers diagram $\mD=(c_1,\ldots,c_n)$ of order $n$ is called \textbf{convex} if, for all $i\in\{1,\ldots,n-1\}$, one has that $c_{i+1}-c_i\leq 1$.
    \end{definition}

    \rev{The reader can easily verify that the empty diagram $\mD=\emptyset$ and the full diagram $\mD=[n]^2$ are both monotone and convex.}

           \begin{example}\label{exa:convex}
    The following are examples of convex Ferrers diagrams of order $n=5$.
  %  \bigskip 
       $$ \mD=(0,1,2,2,3), \;\; \mD=(2,3,3,4,5),\;\; \mD=(3,3,4,5,5),\;\; \mD=(2,3,3,3,4)$$
 
                    
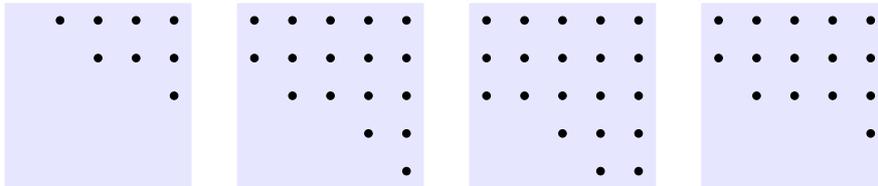
\centering
                           \begin{tikzpicture}[scale=0.5]
\draw[help lines, very thick, white, fill=blue!10] (0,1) -- (0,-4) -- (5,-4)--(5,1)--(0,1);
\draw[fill=black] (1.5,0.5) circle (0.1cm);
\draw[fill=black] (2.5,0.5) circle (0.1cm);
\draw[fill=black] (3.5,0.5) circle (0.1cm);
\draw[fill=black] (4.5,0.5) circle (0.1cm);
\draw[fill=black] (2.5,-0.5) circle (0.1cm);
\draw[fill=black] (3.5,-0.5) circle (0.1cm);
\draw[fill=black] (4.5,-0.5) circle (0.1cm);
\draw[fill=black] (4.5,-1.5) circle (0.1cm);
\end{tikzpicture}  \quad
                    \begin{tikzpicture}[scale=0.5]
\draw[help lines, very thick, white, fill=blue!10] (0,1) -- (0,-4) -- (5,-4)--(5,1)--(0,1);
\draw[fill=black] (0.5,0.5) circle (0.1cm);
\draw[fill=black] (1.5,0.5) circle (0.1cm);
\draw[fill=black] (2.5,0.5) circle (0.1cm);
\draw[fill=black] (3.5,0.5) circle (0.1cm);
\draw[fill=black] (4.5,0.5) circle (0.1cm);
\draw[fill=black] (0.5,-0.5) circle (0.1cm);
\draw[fill=black] (1.5,-0.5) circle (0.1cm);
\draw[fill=black] (2.5,-0.5) circle (0.1cm);
\draw[fill=black] (3.5,-0.5) circle (0.1cm);
\draw[fill=black] (4.5,-0.5) circle (0.1cm);
\draw[fill=black] (1.5,-1.5) circle (0.1cm);
\draw[fill=black] (2.5,-1.5) circle (0.1cm);
\draw[fill=black] (3.5,-1.5) circle (0.1cm);
\draw[fill=black] (4.5,-1.5) circle (0.1cm);
\draw[fill=black] (3.5,-2.5) circle (0.1cm);
\draw[fill=black] (4.5,-2.5) circle (0.1cm);
\draw[fill=black] (4.5,-3.5) circle (0.1cm);
\end{tikzpicture}         \quad           \begin{tikzpicture}[scale=0.5]
\draw[help lines, very thick, white, fill=blue!10] (0,1) -- (0,-4) -- (5,-4)--(5,1)--(0,1);
\draw[fill=black] (0.5,0.5) circle (0.1cm);
\draw[fill=black] (1.5,0.5) circle (0.1cm);
\draw[fill=black] (2.5,0.5) circle (0.1cm);
\draw[fill=black] (3.5,0.5) circle (0.1cm);
\draw[fill=black] (4.5,0.5) circle (0.1cm);
\draw[fill=black] (0.5,-0.5) circle (0.1cm);
\draw[fill=black] (1.5,-0.5) circle (0.1cm);
\draw[fill=black] (2.5,-0.5) circle (0.1cm);
\draw[fill=black] (3.5,-0.5) circle (0.1cm);
\draw[fill=black] (4.5,-0.5) circle (0.1cm);
\draw[fill=black] (0.5,-1.5) circle (0.1cm);
\draw[fill=black] (1.5,-1.5) circle (0.1cm);
\draw[fill=black] (2.5,-1.5) circle (0.1cm);
\draw[fill=black] (3.5,-1.5) circle (0.1cm);
\draw[fill=black] (4.5,-1.5) circle (0.1cm);
\draw[fill=black] (2.5,-2.5) circle (0.1cm);
\draw[fill=black] (3.5,-2.5) circle (0.1cm);
\draw[fill=black] (4.5,-2.5) circle (0.1cm);
\draw[fill=black] (3.5,-3.5) circle (0.1cm);
\draw[fill=black] (4.5,-3.5) circle (0.1cm);
\end{tikzpicture}               \quad           \begin{tikzpicture}[scale=0.5]
\draw[help lines, very thick, white, fill=blue!10] (0,1) -- (0,-4) -- (5,-4)--(5,1)--(0,1);
\draw[fill=black] (0.5,0.5) circle (0.1cm);
\draw[fill=black] (1.5,0.5) circle (0.1cm);
\draw[fill=black] (2.5,0.5) circle (0.1cm);
\draw[fill=black] (3.5,0.5) circle (0.1cm);
\draw[fill=black] (4.5,0.5) circle (0.1cm);
\draw[fill=black] (0.5,-0.5) circle (0.1cm);
\draw[fill=black] (1.5,-0.5) circle (0.1cm);
\draw[fill=black] (2.5,-0.5) circle (0.1cm);
\draw[fill=black] (3.5,-0.5) circle (0.1cm);
\draw[fill=black] (4.5,-0.5) circle (0.1cm);
\draw[fill=black] (1.5,-1.5) circle (0.1cm);
\draw[fill=black] (2.5,-1.5) circle (0.1cm);
\draw[fill=black] (3.5,-1.5) circle (0.1cm);
\draw[fill=black] (4.5,-1.5) circle (0.1cm);
\draw[fill=black] (4.5,-2.5) circle (0.1cm);
\end{tikzpicture}   
\end{example}

We remark that convex Ferrers diagrams are adjoints of monotone Ferrers diagrams and the other way around, as one can see from \cref{exa:monotone} and \cref{exa:convex}.

\begin{definition}\label{def:p-height} 
Let $p$ be a prime number and let $\mD=(c_1,\ldots,c_n)$ be a Ferrers diagram of order $n$. Let $h$ be the largest nonnegative integer with the following properties:
 \begin{enumerate}[label=(\arabic*)]
 \item\label{it:p0} The power $p^h$ divides $n$.
     \item\label{it:p1} For every $i\in\{1,\ldots,n\}$, the power $p^h$ divides $c_i$.
     \item\label{it:p2} If $r,s\in\{0,\ldots,n-1\}$ \rev{are such that there exists  a nonnegative integer $Q$ with the property that $r,s\in\{Qp^h, Qp^h+1, \ldots, (Q+1)p^h-1\}$}, then $c_{n-r}=c_{n-s}$. 
 \end{enumerate}
 The number $h$ is called the \textbf{$p$-height} of $\mD$ and the \textbf{$p$-contraction} of $\mD$ is the Ferrers diagram $\mD^{( p)}=(c_1',\ldots,c_{n/p^{h}}')$  of order $n/p^{h}$ that is defined by 
 $$c_i'=\frac{c_{p^h i}}{p^{h}}, \quad i\in\{1,\ldots, n/p^{h}\}.$$
\end{definition}

We remark that, in the context of \cref{def:p-height}, the value $h=0$ satisfies \ref{it:p0}, \ref{it:p1}, and \ref{it:p2} so the $p$-height of $\mD$ is well defined. 
Moreover, if $\mD$ is a Ferrers diagram of order $n$ of $p$-height $h=0$, then the $p$-contraction of $\mD$ is equal to $\mD$.
Roughly speaking, as the following example suggests, if $h$ is the $p$-height of a Ferrers diagram $\mD$, then $[n]^2$ can be partitioned into blocks of size $p^h\times p^h$ where the intersection of $\mD$ with any block is either full or empty.

\begin{example}\label{exa:2-reduction}
The Ferrers diagram $\mD=(4,4,4,4,8,8,8,8)$ of order $8$ has $2$-height equal to $2$. Its $2$-contraction is the Ferrers diagram $\mD^{(2)}=(1,2)$ of order $2$. 
\medskip
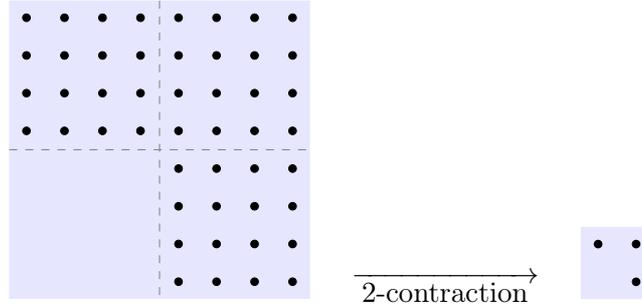
\begin{center}
    \begin{tikzpicture}[scale=0.5]
\draw[help lines, very thick, white, fill=blue!10] (0,1) -- (0,-7) -- (8,-7)--(8,1)--(0,1);
 \draw[help lines, dashed] (4,0.95) -- (4,-6.95);
 \draw[help lines, dashed] (0.05,-3) -- (7.95,-3);
\draw[fill=black] (0.5,0.5) circle (0.1cm);
\draw[fill=black] (1.5,0.5) circle (0.1cm);
\draw[fill=black] (2.5,0.5) circle (0.1cm);
\draw[fill=black] (3.5,0.5) circle (0.1cm);
\draw[fill=black] (4.5,0.5) circle (0.1cm);
\draw[fill=black] (5.5,0.5) circle (0.1cm);
\draw[fill=black] (6.5,0.5) circle (0.1cm);
\draw[fill=black] (7.5,0.5) circle (0.1cm);

\draw[fill=black] (0.5,-0.5) circle (0.1cm);
\draw[fill=black] (1.5,-0.5) circle (0.1cm);
\draw[fill=black] (2.5,-0.5) circle (0.1cm);
\draw[fill=black] (3.5,-0.5) circle (0.1cm);
\draw[fill=black] (4.5,-0.5) circle (0.1cm);
\draw[fill=black] (5.5,-0.5) circle (0.1cm);
\draw[fill=black] (6.5,-0.5) circle (0.1cm);
\draw[fill=black] (7.5,-0.5) circle (0.1cm);

\draw[fill=black] (0.5,-1.5) circle (0.1cm);
\draw[fill=black] (1.5,-1.5) circle (0.1cm);
\draw[fill=black] (2.5,-1.5) circle (0.1cm);
\draw[fill=black] (3.5,-1.5) circle (0.1cm);
\draw[fill=black] (4.5,-1.5) circle (0.1cm);
\draw[fill=black] (5.5,-1.5) circle (0.1cm);
\draw[fill=black] (6.5,-1.5) circle (0.1cm);
\draw[fill=black] (7.5,-1.5) circle (0.1cm);

\draw[fill=black] (0.5,-2.5) circle (0.1cm);
\draw[fill=black] (1.5,-2.5) circle (0.1cm);
\draw[fill=black] (2.5,-2.5) circle (0.1cm);
\draw[fill=black] (3.5,-2.5) circle (0.1cm);
\draw[fill=black] (4.5,-2.5) circle (0.1cm);
\draw[fill=black] (5.5,-2.5) circle (0.1cm);
\draw[fill=black] (6.5,-2.5) circle (0.1cm);
\draw[fill=black] (7.5,-2.5) circle (0.1cm);

\draw[fill=black] (4.5,-3.5) circle (0.1cm);
\draw[fill=black] (5.5,-3.5) circle (0.1cm);
\draw[fill=black] (6.5,-3.5) circle (0.1cm);
\draw[fill=black] (7.5,-3.5) circle (0.1cm);

\draw[fill=black] (4.5,-4.5) circle (0.1cm);
\draw[fill=black] (5.5,-4.5) circle (0.1cm);
\draw[fill=black] (6.5,-4.5) circle (0.1cm);
\draw[fill=black] (7.5,-4.5) circle (0.1cm);

\draw[fill=black] (4.5,-5.5) circle (0.1cm);
\draw[fill=black] (5.5,-5.5) circle (0.1cm);
\draw[fill=black] (6.5,-5.5) circle (0.1cm);
\draw[fill=black] (7.5,-5.5) circle (0.1cm);

\draw[fill=black] (4.5,-6.5) circle (0.1cm);
\draw[fill=black] (5.5,-6.5) circle (0.1cm);
\draw[fill=black] (6.5,-6.5) circle (0.1cm);
\draw[fill=black] (7.5,-6.5) circle (0.1cm);
\end{tikzpicture}  \quad  $\overrightarrow{\ 2\mbox{-contraction }}$ \quad
    \begin{tikzpicture}[scale=0.5]
\draw[help lines, very thick, white, fill=blue!10] (0,1) -- (0,-1) -- (2,-1)--(2,1)--(0,1);
\draw[fill=black] (0.5,0.5) circle (0.1cm);
\draw[fill=black] (1.5,0.5) circle (0.1cm);

\draw[fill=black] (1.5,-0.5) circle (0.1cm);

\end{tikzpicture} 
\end{center}
\end{example}

\begin{definition}\label{def:p-monotone}
Let $p$ be a prime number and let $\mD=(c_1,\ldots,c_n)$ be a Ferrers diagram of order $n$. The diagram $\mD$ is called \textbf{$p$-monotone} if its $p$-contraction is monotone. 
    Similarly, $\mD$ is \textbf{$p$-convex} if its $p$-contraction is convex.
\end{definition}

We remark that, with this new definition, monotone Ferrers diagrams are a special instance of the larger family of $p$-monotone Ferrers diagrams. 

\begin{example} The following are examples of $2$-monotone Ferrers diagrams of order $8$.
$$\mD=(0,0,2,2),\qquad \mD=(2,2,4,4,6,6), \qquad \mD=(0,0,4,4,8,8,8,8,8)
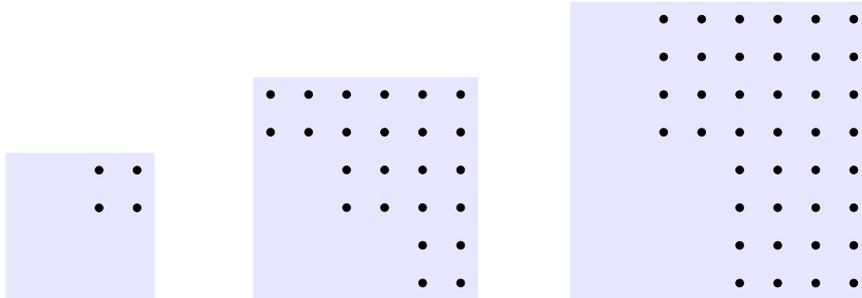
$$

   \centering    
   
   \begin{tikzpicture}[scale=0.5]
\draw[help lines, very thick, white, fill=blue!10] (0,1) -- (0,-3) -- (4,-3)--(4,1)--(0,1);

\draw[fill=black] (2.5,0.5) circle (0.1cm);
\draw[fill=black] (3.5,0.5) circle (0.1cm);

\draw[fill=black] (2.5,-0.5) circle (0.1cm);
\draw[fill=black] (3.5,-0.5) circle (0.1cm);

\end{tikzpicture} \;\;\;\qquad 
   \begin{tikzpicture}[scale=0.5]
\draw[help lines, very thick, white, fill=blue!10] (0,1) -- (0,-5) -- (6,-5)--(6,1)--(0,1);
\draw[fill=black] (0.5,0.5) circle (0.1cm);
\draw[fill=black] (1.5,0.5) circle (0.1cm);
\draw[fill=black] (2.5,0.5) circle (0.1cm);
\draw[fill=black] (3.5,0.5) circle (0.1cm);
\draw[fill=black] (4.5,0.5) circle (0.1cm);
\draw[fill=black] (5.5,0.5) circle (0.1cm);

\draw[fill=black] (0.5,-0.5) circle (0.1cm);
\draw[fill=black] (1.5,-0.5) circle (0.1cm);
\draw[fill=black] (2.5,-0.5) circle (0.1cm);
\draw[fill=black] (3.5,-0.5) circle (0.1cm);
\draw[fill=black] (4.5,-0.5) circle (0.1cm);
\draw[fill=black] (5.5,-0.5) circle (0.1cm);

\draw[fill=black] (2.5,-1.5) circle (0.1cm);
\draw[fill=black] (3.5,-1.5) circle (0.1cm);
\draw[fill=black] (4.5,-1.5) circle (0.1cm);
\draw[fill=black] (5.5,-1.5) circle (0.1cm);

\draw[fill=black] (2.5,-2.5) circle (0.1cm);
\draw[fill=black] (3.5,-2.5) circle (0.1cm);
\draw[fill=black] (4.5,-2.5) circle (0.1cm);
\draw[fill=black] (5.5,-2.5) circle (0.1cm);

\draw[fill=black] (4.5,-3.5) circle (0.1cm);
\draw[fill=black] (5.5,-3.5) circle (0.1cm);

\draw[fill=black] (4.5,-4.5) circle (0.1cm);
\draw[fill=black] (5.5,-4.5) circle (0.1cm);

\end{tikzpicture}  \; \qquad 
   \begin{tikzpicture}[scale=0.5]
\draw[help lines, very thick, white, fill=blue!10] (0,1) -- (0,-7) -- (8,-7)--(8,1)--(0,1);

\draw[fill=black] (2.5,0.5) circle (0.1cm);
\draw[fill=black] (3.5,0.5) circle (0.1cm);
\draw[fill=black] (4.5,0.5) circle (0.1cm);
\draw[fill=black] (5.5,0.5) circle (0.1cm);
\draw[fill=black] (6.5,0.5) circle (0.1cm);
\draw[fill=black] (7.5,0.5) circle (0.1cm);

\draw[fill=black] (2.5,-0.5) circle (0.1cm);
\draw[fill=black] (3.5,-0.5) circle (0.1cm);
\draw[fill=black] (4.5,-0.5) circle (0.1cm);
\draw[fill=black] (5.5,-0.5) circle (0.1cm);
\draw[fill=black] (6.5,-0.5) circle (0.1cm);
\draw[fill=black] (7.5,-0.5) circle (0.1cm);

\draw[fill=black] (2.5,-1.5) circle (0.1cm);
\draw[fill=black] (3.5,-1.5) circle (0.1cm);
\draw[fill=black] (4.5,-1.5) circle (0.1cm);
\draw[fill=black] (5.5,-1.5) circle (0.1cm);
\draw[fill=black] (6.5,-1.5) circle (0.1cm);
\draw[fill=black] (7.5,-1.5) circle (0.1cm);

\draw[fill=black] (2.5,-2.5) circle (0.1cm);
\draw[fill=black] (3.5,-2.5) circle (0.1cm);
\draw[fill=black] (4.5,-2.5) circle (0.1cm);
\draw[fill=black] (5.5,-2.5) circle (0.1cm);
\draw[fill=black] (6.5,-2.5) circle (0.1cm);
\draw[fill=black] (7.5,-2.5) circle (0.1cm);

\draw[fill=black] (4.5,-3.5) circle (0.1cm);
\draw[fill=black] (5.5,-3.5) circle (0.1cm);
\draw[fill=black] (6.5,-3.5) circle (0.1cm);
\draw[fill=black] (7.5,-3.5) circle (0.1cm);

\draw[fill=black] (4.5,-4.5) circle (0.1cm);
\draw[fill=black] (5.5,-4.5) circle (0.1cm);
\draw[fill=black] (6.5,-4.5) circle (0.1cm);
\draw[fill=black] (7.5,-4.5) circle (0.1cm);

\draw[fill=black] (4.5,-5.5) circle (0.1cm);
\draw[fill=black] (5.5,-5.5) circle (0.1cm);
\draw[fill=black] (6.5,-5.5) circle (0.1cm);
\draw[fill=black] (7.5,-5.5) circle (0.1cm);

\draw[fill=black] (4.5,-6.5) circle (0.1cm);
\draw[fill=black] (5.5,-6.5) circle (0.1cm);
\draw[fill=black] (6.5,-6.5) circle (0.1cm);
\draw[fill=black] (7.5,-6.5) circle (0.1cm);
\end{tikzpicture}  
\end{example}

The following straightforward result \rev{ensures} that  $p$-convex Ferrers diagrams and $p$-monotone Ferrers diagrams are adjoint to each other. 

\begin{lemma}\label{lemma:adj}
Let $p$ be a prime number and let $\mD=(c_1,\ldots,c_n)$ be a Ferrers diagram of order $n$. \rev{Then the following hold:
\begin{enumerate}[label=$(\arabic*)$]
    \item The diagrams $\mD$ and $\mD^{\top}$ have the same $p$-height.
    \item One has $(\mD^{\top})^{(p)}=(\mD^{(p)})^{\top}$.
    \item The diagram $\mD$ is $p$-monotone if and only if $\mD^\top$ is $p$-convex.
\end{enumerate}

}
\end{lemma}

In the next section, we will describe a very useful representation of Ferrers diagram matrix spaces $\F^\mD$ when $\mD$ is $p$-monotone. Furthermore, by adjunction -- that is, transposing with respect to the antidiagonal -- we automatically obtain also a representation of $\F^\mD$ when $\mD$ is $p$-convex. 

The following remark generalizes \cref{rem:monotonicity} to the case of $p$-monotone Ferrers diagrams. 

    \begin{remark}\label{rem:p-monotonicity}
        Let  $\mD=(c_1,\ldots,c_n)$ be a $p$-monotone Ferrers diagram of order $n$ with $p$-height equal to $h$. Set 
        $\ell=\min\{i \,:\, c_i \neq 0\}$ and $r=\max\{ i \,:\, c_i\neq n\}.$ 
        It follows from $p$-monotonicity that $$\ell=1+s p^h \ \ \textup{and} \ \ r=t p^h, \ \ \textup{ for some }s,t\in\{0,\ldots,n/p^h\}.$$ 
        The Ferrers diagram $\mD$ being \rev{$p$-monotone}, we derive that, if $a,b\in\{s,\ldots, t+1\}$, then 
        \[
        a>b \ \ \Longrightarrow \ \ c_{ap^h}-c_{bp^h}\geq (a-b)p^h.
        \]
    \end{remark}

    \begin{remark}\label{rem:deletion-monotone}
    We note that, in the special case where $\mD$ is a $p$-monotone Ferrers diagram, \rev{for every choice of an integer $1\leq d\leq n$,
    one has that $\nu_0(\mD,d)\geq \nu_1(\mD,d)\geq\ldots\geq\nu_{d-1}(\mD,d)$
     and thus the value of $\nu_{\min}(\mD,d)$ can be rewritten as} 
    \[
    \rev{\nu_{\min}(\mD,d)=\min\{\nu_j(\mD,d)\ :\ j\in\{0,\ldots,d-1\}\}=\nu_{d-1}(\mD,d)=\sum_{i=1}^{n-d+1}c_i.}
    \]
    In other words, $\nu_{\min}(\mD,d)$ can be computed by exclusively deleting columns from the Ferrers diagram $\mD$. This is a direct consequence of \cref{rem:p-monotonicity}.
\end{remark}

\section{Flags and modular skew algebras}\label{sec:modular}

In this section we fix a prime $p$, a positive integer  $n=p^m$, and a field $\F$ of characteristic $p$ with a cyclic Galois extension $\LL$ of degree $n$. 
We fix a generator $\sigma$ of $\Gal(\LL/\F)$, and define $\bsig=\sigma-\id\in\LL[\sigma]$. \rev{Since the characteristic of the field divides the order of $\sigma$, resembling the classical group algebra setting, the skew algebra $\LL[\sigma]$ is called \textbf{modular}.} Note that, in this case the endomorphism $\bsig$ is nilpotent of order $n$, since 
$$\bsig^n=(\sigma-\id)^n=(\sigma-\id)^{p^m}=\sigma^{p^m}-\id,$$
which maps every element of $\LL$ to $0$. 
We define
\begin{equation}\label{eq:flag} \mF_i=\ker(\bsig^i), \quad \mbox { for every } i\in\{0,\ldots,n\}.
\end{equation}
In particular, we have that $\mF_0=\{0\}$, $\mF_1=\F$, and $\mF_n=\LL$.
We work under these assumptions until the end of the present section. The following result collects some immediate properties of the endomorphisms $\bsig^i$ and of the $\F$-subspaces $\mF_i$ of $\LL$.  

\begin{lemma}\label{lem:prop_bsig}
The following hold:
\begin{enumerate}[label=$(\arabic*)$]
 \item The set $\{\bsig^i : i \in\{0,\ldots,n-1\}\}$ is an $\LL$-basis of $\LL[\sigma]$.
 \item For each $i \in \{0,\ldots,n\}$, one has $\dim_\F\mF_i=i$.
 \item\label{it:3} For each $i \in \{1,\ldots,n\}$, one has $\mF_{i-1}\subseteq \mF_i$.
 \item\label{it:4} For every $i,j\in\{0,\ldots,n\}$, one has $\bsig^j(\mF_i)=\mF_{\max\{0,i-j\}}$.
\end{enumerate}
\end{lemma}

A direct consequence of Lemma \ref{lem:prop_bsig} is that $\mF=(\mF_0,\ldots,\mF_n)$ is a \textbf{full flag} of the  $\F$-vector space $\LL$, that is, $\{0\}=\mF_0\subset \mF_1\subset \ldots \subset \mF_{n-1}\subset \mF_n=\LL$ is a chain of $\F$-vector spaces and the $\F$-dimension of each quotient $\mF_i/\mF_{i-1}$ is equal to $1$.

\begin{definition} An $\F$-basis $\mB=(\beta_1,\ldots,\beta_n)$ of $\LL$  is called \textbf{$\mF$-compatible} if, for every index $i\in\{1,\ldots, n\}$, the $\F$-span $\gen{\beta_1,\ldots,\beta_i}_\F$ is equal to $\mF_i$.
\end{definition}

The following result on the multiplicative behaviour of the flag $\mF$ will be crucial in the proof of \cref{thm:monotone_representation}.

\begin{proposition}\label{prop:absorbing}
Let $a,b,h\geq 0$ be integers satisfying  $p^h(a+b)\leq n$. Then one has
    $$ \mF_{p^ha}\cdot \mF_{p^hb}\subseteq \mF_{p^h(a+b-1)}.$$
\end{proposition}

\begin{proof}
We work by induction on $a+b$. If $a=0$ or $b=0$, then the  claim is true since $\mF_0=\{0\}$. We assume thus that $a,b\geq 1$.  Let $x\in \mF_{p^ha}$ and $y\in\mF_{p^hb}$ be arbitrary and write \rev{
$$u=\bsig^{p^h}(x)=\sigma^{p^h}(x)-x \ \ \textup{ and } \ \ v=\bsig^{p^h}(y)=\sigma^{p^h}(y)-y.$$ 
\cref{lem:prop_bsig}\ref{it:4} implies that $u\in\mF_{p^h(a-1)}$ and $v\in\mF_{p^h(b-1)}$}.
We compute
   \[
   \sigma^{p^h}(xy)=\sigma^{p^h}(x)\sigma^{p^h}(y)=(x+u)(y+v)=xy+xv+yu+uv,
   \]
   from which we derive that 
   \[
   \bsig^{p^h}(xy)=
   \sigma^{p^h}(xy)-xy=xv+yu+uv\in \mF_{p^ha}\cdot\mF_{p^h(b-1)}+\mF_{p^hb}\cdot\mF_{p^h(a-1)}+\mF_{p^h(a-1)}\cdot\mF_{p^h(b-1)}.
   \]
   It follows from the induction hypothesis that $\bsig^{p^h}(xy)\in\mF_{p^h(a+b-2)}$ and so $xy\in\mF_{p^h(a+b-1)}$.
\end{proof}

    Let $\mD=(c_1,\ldots,c_n)$ be a Ferrers diagram of order $n=p^m$. Throughout we write $\LL[\sigma;\mD]$  for the $\F$-subspace 
    \[
    \LL[\sigma; \mD]=\bigoplus_{i=1}^n \mF_{c_i}\bsig^{i-1}= \left\{\sum_{i=1}^{n} \lambda_i \bsig^{i-1} \,:\, \lambda_i\in \mF_{c_{i}}\right\}
    \]
    of $\LL[\sigma]$ defined by $\mD$ with respect to the flag $\mF$.

   \begin{theorem}\label{thm:monotone_representation}
        Let $\mD=(c_1,\ldots,c_n)$ be a  $p$-monotone Ferrers diagram of order $n=p^m$. Let $\LL$ be a cyclic Galois extension of $\F$ of order $n$, whose Galois group is generated by $\sigma$. 
        \rev{Write $\mF=(\mF_0,\ldots,\mF_n)$, where $\mF_i$ is as in \eqref{eq:flag}, and let $\mB$ be an $\mF$-compatible $\F$-basis of $\LL$. Then, the map $\phi_{\mathcal{B}}$ in \eqref{eq:matrix_isomorphism} maps $\LL[\sigma;\mD]$ isomorphically into $\F^\mD$.} 
    \end{theorem}

    \begin{proof}
       \rev{ Since $\phi_{\mathcal{B}}$ is injective, the $\F$-vector space $\LL[\sigma;\mD]$ is mapped to a space of the same $\F$-dimension. Now, $\dim_\F(\LL[\sigma;\mD])=c_1+\ldots+c_n=\dim_\F(\F^\mD)$, so to prove that $\phi_{\mathcal{B}}(\LL[\sigma;\mD])=\F^\mD$, it suffices to show that $\phi_{\mathcal{B}}(\LL[\sigma;\mD])\subseteq\F^\mD$. Doing this is equivalent to proving that }
        \begin{equation}\label{eq:containments}
            f(\mF_{i})\subseteq \mF_{c_{i}} \qquad \quad \mbox{ for every } f\in \LL[\sigma;\mD] \textup{ and } i\in\{1,\ldots,n\}.
        \end{equation}
        We proceed thus to prove \eqref{eq:containments}.
To this end, let $h$ be the $p$-height of $\mD$. This means that $p^h$ divides every entry $c_k$ of $\mD$ and, for every $k\in\{0,\ldots, n/p^h-1\}$, that $c_{1+kp^h}=c_{2+kp^h}=\ldots =c_{(k+1)p^h}$. 
Since, for each index $k$, one has 
\[
\mF_{1+kp^h}\subset \ldots \subset \mF_{(k+1)p^h} \ \ \textup{ and } \ \ \mF_{c_{1+kp^h}}=\ldots=\mF_{c_{(k+1)p^h}},
\]
checking \eqref{eq:containments} is equivalent to checking \eqref{eq:containments} for indices $i$ of the form $i=ap^h$ with $a\in\{1,\ldots n/p^h\}$. 
Let thus $a\in\{1,\ldots n/p^h\}$ and write $i=ap^h$. Let, moreover, $f\in \LL[\sigma;\mD]$: 
leveraging on linearity, we assume without loss of generality that $f=\lambda_{j} \bsig^{j-1}$, for some $j\in\{1,\ldots,n\}$ and $\lambda_{j}\in\mF_{c_j}$. Applying the same argument as before, we assume without loss of generality that $j=bp^h$ with $b\in\{1,\ldots, n/p^h\}$. Let $s$ and $t$ be as in \cref{rem:p-monotonicity}. 

If $b < s$, then $\mF_{j}=\{0\}$ and we derive trivially $f(\mF_i)=\{0\}\subseteq \mF_{c_i}$. In a dual way, if $a>t$, then $\mF_{c_{i}}=\LL$  and thus we have $f(\mF_i)\subseteq \mF_{c_i}$. Moreover, if $a\leq b$, then Lemma \ref{lem:prop_bsig}\ref{it:4} yields that $\bsig^{j-1}(\mF_{i})=\{0\}$ and, once again, we get $f(\mF_i)\subseteq \mF_{c_i}$.

We assume now that $s\leq b<a\leq t$ and write $c_i=cp^h$ and $c_j=dp^h$. \cref{rem:p-monotonicity} ensures that $c-d\geq a-b$ and so we have $p^h(d+a-b+1)\leq p^h(c+1)\leq n$.
With the aid of \cref{lem:prop_bsig}\ref{it:4} and \cref{prop:absorbing} we compute 
\[
f(\mF_i)=\lambda_j\bsig^{j-1}(\mF_{i})=\lambda_j\mF_{i-j+1}=\lambda_j(\mF_{ap^h-bp^h+1})\subseteq \mF_{dp^h}\mF_{(a-b+1)p^h}\subseteq \mF_{(a-b+d)p^h}.  
\]
 and thus $f(\mF_i)\subseteq \mF_{(a-b+d)p^h}\subseteq \mF_{cp^h}=\mF_{c_i}$.
    \end{proof}

 We conclude this section by providing a few  concrete examples.

\begin{example}\label{exa:representation_n=5}
Let $\F=\F_5$ and $n=5$. Consider the monotone Ferrers diagram $\mD=(1,3,4,5,5)$ and the space $\F^{\mD}$ representing the  matrices supported on $\mD$ over $\F=\F_5$.
    Furthermore, consider the degree $5$ extension field $\LL=\F_{5^5}=\F_5(\gamma)$, where $\gamma^5+4\gamma+3=0$. Let $\sigma$ be the $5$-Frobenius automorphism of $\F_{5^5}$ defined as $\sigma(\alpha)=\alpha^5$ and set $\bsig=\sigma-\id$. Take the full flag $\mF=(\mF_0,\mF_1,\mF_2,\mF_3,\mF_4,\mF_5)$ of $\F_{5^5}$ over $\F_5$ given by
    $$\mF_i=\ker(\bsig^i).$$      
    It is easy to see that $\mB=(1,\gamma^{2968},\gamma^{1531},\gamma^{1556},\gamma^{1566})$ is an $\mF$-compatible $\F_5$-basis of $\F_{5^5}$.
     By Theorem \ref{thm:monotone_representation}, the algebra $\F_5^{\mD}$ is then isomorphic to 
    $$\F_{5^5}[\sigma;\mD]= \mF_1\id \oplus \, \mF_3\bsig\oplus \mF_4\bsig^2\oplus \mF_5\bsig^3\oplus \mF_5\bsig^4,$$
    where the isomorphism is with respect to the chosen $\F_5$-basis $\mB$. 
\end{example}

\begin{example}\label{exa:n=8_UT}
    Let $\F=\F_2$ and $n=8$. Let, moreover, $\mD=\mT_8=(1,2,3,4,5,6,7,8)$ and write $\F^{\mD}$ for the space representing the  upper triangular $8\times 8$ matrices over $\F=\F_2$. Furthermore, consider the degree $8$ extension field $\LL=\F_{2^8}=\F_2(\gamma)$, where $\gamma^8+\gamma^4+\gamma^3+\gamma^2+1=0$. Let $\sigma:\F_{2^8}\rightarrow \F_{2^8}$ be the $2$-Frobenius automorphism which is defined as $\sigma(\alpha)=\alpha^2$ and let $\bsig=\sigma-\id$. As in \cref{exa:representation_n=5} let $\mF=(\mF_0,\ldots,\mF_8)$ be the flag in the $\F_2$-vector space $\F_{2^8}$ given by 
$$\mF_i=\ker(\bsig^i).$$
    Then an $\mF$-compatible $\F_2$-basis of $\F_{2^8}$ is given by $\mB=(1,\gamma^{170},\gamma^{136},\gamma^{204},\gamma^{222},\gamma^{38},\gamma^{143},\gamma^{5})$. By Theorem \ref{thm:monotone_representation}, the algebra $\F_2^{\mD}$ is isomorphic to 
    $$\F_{2^8}[\sigma;\mD]=\left\{\sum_{i=1}^8 \lambda_i\bsig^{i-1} \,:\, \lambda_i \in \mF_i\right\},$$
    where the isomorphism depends on  the $\F_2$-basis $\mB$. 
    Let, for instance, $f=\id+\gamma^{68}\bsig^2\in \F_{2^8}[\sigma;\mD]$:  a straightforward computation reveals that the matrix representation of $f$ with respect to $\mB$ is
    $$\left(\begin{array}{ccccccccc}
         1 \cellcolor{blue!10} & 0 \cellcolor{blue!10} & 1 \cellcolor{blue!10} & 1 \cellcolor{blue!10} & 0 \cellcolor{blue!10} & 0 \cellcolor{blue!10} & 1 \cellcolor{blue!10} & 0 \cellcolor{blue!10} \\
         0 & 1 \cellcolor{blue!10} & 1 \cellcolor{blue!10} & 1 \cellcolor{blue!10} & 0 \cellcolor{blue!10} & 1 \cellcolor{blue!10} & 1 \cellcolor{blue!10} & 1 \cellcolor{blue!10} \\
         0 & 0 & 0 \cellcolor{blue!10} & 0 \cellcolor{blue!10} & 0 \cellcolor{blue!10} & 1 \cellcolor{blue!10} & 0 \cellcolor{blue!10} & 0 \cellcolor{blue!10} \\
         0 & 0 & 0 & 0 \cellcolor{blue!10} & 1 \cellcolor{blue!10} & 0 \cellcolor{blue!10} & 0 \cellcolor{blue!10} & 1 \cellcolor{blue!10} \\
         0 & 0 & 0 & 0 & 1 \cellcolor{blue!10} & 0 \cellcolor{blue!10} & 1 \cellcolor{blue!10} & 1 \cellcolor{blue!10} \\
         0 & 0 & 0 & 0 & 0 & 1 \cellcolor{blue!10} & 1 \cellcolor{blue!10} & 1 \cellcolor{blue!10} \\
         0 & 0 & 0 & 0 & 0 & 0 & 0 \cellcolor{blue!10} & 0 \cellcolor{blue!10} \\
         0 & 0 & 0 & 0 & 0 & 0 & 0 & 0 \cellcolor{blue!10}
    \end{array}\right) \in \F_2^{\mD}.$$
\end{example}

\begin{example} Let $\mD=(4,4,4,4,8,8,8,8)$ be the $2$-monotone Ferrers diagram of order $8$ from Example \ref{exa:2-reduction}. With the same notation used in Example \ref{exa:n=8_UT}, the space $\F_2^{\mD}$ is isomorphic to 
$$\F_{2^8}[\sigma;\mD]=\left\{\sum_{i=1}^8\lambda_i \bsig^{i-1} \,:\, \lambda_1,\lambda_2,\lambda_3,\lambda_4 \in \mF_4, \lambda_5,\lambda_6,\lambda_7,\lambda_8 \in \mF_8,\right\},$$
where the isomorphism is induced by the $\F_2$-basis $\mB=(1,\gamma^{170},\gamma^{136},\gamma^{204},\gamma^{222},\gamma^{38},\gamma^{143},\gamma^{5})$.
\end{example}

    \section{Ferrers diagram rank-metric codes}\label{sec:ferrers}

In this section we prove our main results: \cref{thm:MFD-monotone}, \cref{th:strictly-monotone}, and \cref{thm:MDS-constr-column}. We  deal with the first two via the construction of MFD codes from \cref{thm:monotone_representation}, relying on \eqref{eq:bound_degree}. 
 We prove the third theorem with the aid of \cref{th:strictly-monotone} and \cref{prop:MFD-from-old}.

    \subsection{Monotone and convex Ferrers diagrams}

    In this section we give a constructive proof of Conjecture \ref{conj:ES} for $p$-monotone Ferrers diagrams and their adjoints over \rev{finite} fields of characteristic $p$. To this end, we work under the same assumptions and use the same notation as in \cref{sec:modular}. For an integer $1\leq d\leq n$ and a Ferrers diagram $\mD=(c_1,\ldots,c_n)$ of order $n=p^m$, we write additionally $\LL[\sigma; \mD]_{n-d}$  for the $\F$-subspace 
    \[
    \LL[\sigma; \mD]_{n-d}=\LL[\sigma]_{n-d}\cap\LL[\sigma;\mD]=\bigoplus_{i=1}^{n-d} \mF_{c_i}\bsig^{i-1}= \left\{\sum_{i=1}^{n-d} \lambda_i \bsig^{i-1} \,:\, \lambda_i\in \mF_{c_{i}}\right\}
    \]
    of polynomials in $\LL[\sigma;\mD]$ of degree at most $n-d$. %The following result is stated implicitly assuming the choice of a basis $\mathcal{B}$ of $\LL$ over $\F$.

        \begin{theorem}\label{thm:construction_monotone}
        Let $\F$ be a finite field, of characteristic $p$ and order $q$, and let $\mD=(c_1,\ldots,c_n)$ be a $p$-monotone Ferrers diagram of order $n=p^m$. Let $\LL$ be an extension of $\F$ of degree $n$ and let $\sigma$ be the $q$-Frobenius automorphism of $\LL$. Let $d$ be an integer with $1\le d\le n$ and write $\bsig=\sigma-\id$.
        \rev{Write $\mF=(\mF_0,\ldots,\mF_n)$, where $\mF_i$ is as in \eqref{eq:flag}, and let $\mB$ be an $\mF$-compatible $\F$-basis of $\LL$. Then the map $\phi_{\mathcal{B}}$ in \eqref{eq:matrix_isomorphism} maps $\LL[\sigma; \mD]_{n-d}$ isomorphically into a $[\mD,\nu_{\min}(\mD,d),d]_{\F}$ MFD code.}
    \end{theorem}

    \begin{proof}
    Thanks to \cref{thm:monotone_representation}, the algebra $\LL[\sigma;\mD]$ maps isomorphically onto $\F^\mD$ via $\phi_{\mathcal{B}}$. Moreover, it follows from \eqref{eq:bound_degree} that every \rev{nonzero} element of $\LL[\sigma; \mD]_{n-d}$ is mapped to a matrix of rank at least $d$. Computing the $\F$-dimension of $\LL[\sigma; \mD]_{n-d}$ we derive from \cref{lem:prop_bsig} and \cref{rem:deletion-monotone} that
    \[
    \dim_{\F}(\LL[\sigma; \mD]_{n-d})=\sum_{i=1}^{n-d}\dim_{\F}\mF_{c_i}=\sum_{i=1}^{n-d}c_i=\nu_{\min}(\mD,d).
    \]
    This concludes the proof.
    \end{proof}

    We derive the following result as an immediate corollary of \cref{thm:construction_monotone} and \cref{lemma:adj}.

\begin{theorem}\label{thm:MFD-monotone}
Let $d,m$ be positive integers and let $p$ be a prime number. Write $n=p^m$ and let $1\le d \le n$. Then the following hold: 
    \begin{enumerate}[label=(\arabic*)]
    \item Conjecture \ref{conj:ES} holds true for $p$-monotone Ferrers diagrams of order $n$ over all finite fields of characteristic $p$.
    \item Conjecture \ref{conj:ES} holds true for $p$-convex Ferrers diagrams of order $n$ over all finite fields of characteristic $p$.
    \end{enumerate}
\end{theorem}

We illustrate Theorem \ref{thm:MFD-monotone} with a concrete example.

\begin{example}
 Let $\F=\F_5$, $n=5$ and $d=4$. Let $\mD=(1,3,4,5,5)$ be the monotone Ferrers diagram  from Example \ref{exa:representation_n=5}. In this case $\nu_{\min}(\mD,4)=4$, and Theorem \ref{thm:MFD-monotone} ensures the existence of a  $[\mD,4,4]_{\F_5}$ MFD code. A concrete example is given by:
 $$\F_{5^5}[\sigma;\mD]_1=\mF_1\id\oplus\mF_3\bsig=\langle \id,\bsig,\gamma^{2968}\bsig,\gamma^{1531}\bsig\rangle_{\F_5}.$$
 Using the $\F_5$-basis $\mB=(1,\gamma^{2968},\gamma^{1531},\gamma^{1556},\gamma^{1566})$ of $\F_{5^5}$ given in Example \ref{exa:representation_n=5}, we obtain that the $[\mD,4,4]_{\F_5}$ MFD code in $\F_5^{\mD}$ is the $\F_5$-subspace generated by the following matrices:
$$\left(\begin{array}{ccccc} 
1\cellcolor{blue!10} & 0\cellcolor{blue!10} & 0\cellcolor{blue!10} & 0\cellcolor{blue!10} & 0\cellcolor{blue!10} \\
0 & 1\cellcolor{blue!10} & 0\cellcolor{blue!10} & 0\cellcolor{blue!10} & 0\cellcolor{blue!10} \\
0 & 0\cellcolor{blue!10} & 1\cellcolor{blue!10} & 0\cellcolor{blue!10} & 0\cellcolor{blue!10} \\
0 & 0 & 0\cellcolor{blue!10} & 1\cellcolor{blue!10} & 0\cellcolor{blue!10} \\
0 & 0 & 0 & 0 \cellcolor{blue!10}& 1\cellcolor{blue!10} 
\end{array}\right),\,
\left(\begin{array}{ccccc}
0\cellcolor{blue!10} & 1\cellcolor{blue!10} & 0\cellcolor{blue!10} & 0\cellcolor{blue!10} & 0\cellcolor{blue!10} \\
0 & 0\cellcolor{blue!10} & 1 \cellcolor{blue!10}& 0 \cellcolor{blue!10}& 0 \cellcolor{blue!10}\\
0 & 0\cellcolor{blue!10} & 0\cellcolor{blue!10} & 1\cellcolor{blue!10} & 0 \cellcolor{blue!10}\\
0 & 0 & 0\cellcolor{blue!10} & 0\cellcolor{blue!10} & 1\cellcolor{blue!10} \\
0 & 0 & 0 & 0 \cellcolor{blue!10}& 0\cellcolor{blue!10} 
\end{array}\right),\,
\left(\begin{array}{ccccc} 
   0 \cellcolor{blue!10}   &  0 \cellcolor{blue!10}  &   0 \cellcolor{blue!10}  &   0  \cellcolor{blue!10} &   4\cellcolor{blue!10} \\
   0    &  1 \cellcolor{blue!10}  &   1 \cellcolor{blue!10}  &   0 \cellcolor{blue!10}  &   0\cellcolor{blue!10} \\
   0    &  0  \cellcolor{blue!10} &   2 \cellcolor{blue!10}  &   2 \cellcolor{blue!10}  &   0\cellcolor{blue!10} \\
   0    &  0   &   0 \cellcolor{blue!10}  &   3 \cellcolor{blue!10}  &   3\cellcolor{blue!10} \\
   0    &  0   &   0   &   0  \cellcolor{blue!10} &   4\cellcolor{blue!10}
\end{array}\right),\,
\left(\begin{array}{ccccc} 
    0  \cellcolor{blue!10} &   0 \cellcolor{blue!10}  &   0 \cellcolor{blue!10}  &   1 \cellcolor{blue!10}  &   0 \cellcolor{blue!10}\\
    0   &   0  \cellcolor{blue!10} &   0  \cellcolor{blue!10} &   0  \cellcolor{blue!10} &   0 \cellcolor{blue!10} \\
    0   &   1 \cellcolor{blue!10}  &   2  \cellcolor{blue!10} &   1 \cellcolor{blue!10}  &   0\cellcolor{blue!10} \\
    0   &   0   &   3 \cellcolor{blue!10}  &   1 \cellcolor{blue!10}  &   3\cellcolor{blue!10} \\
    0   &   0   &   0   &   1 \cellcolor{blue!10}  &   2\cellcolor{blue!10} \\
\end{array}\right).
$$\end{example}

\subsection{Strictly monotone and initially convex Ferrers diagrams}

Among monotone Ferrers diagrams, a distinguished subfamily  is given by strictly monotone ones, whose adjoints are the initially convex Ferrers diagrams.

\begin{definition}\label{def:strictmonotone_initiallyconvex}
    A Ferrers diagram $\mD=(c_1,\ldots,c_n)$ is called \textbf{strictly monotone} if, whenever $c_i>0$, one has $c_{i+1}>c_i$.
    A Ferrers diagram $\mD=(c_1,\ldots,c_n)$ is called \textbf{initially convex} if, for every $i\in\{1,\ldots,n-1\}$, one has $c_1\leq 1$ and $c_{i+1}-c_i\leq 1$.
\end{definition}

It turns out that a Ferrers diagram is initially convex if and only if it is \emph{in canonical form} in the sense of  \cite[Definition 3.6]{cotardo2023diagonals}. Although the two notions coincide, they appear
in different contexts and for apparently unrelated purposes. \rev{We remark that, in contrast with the case of monotone and convex Ferrers diagrams, the empty diagram $\mD=\emptyset$ is strictly monotone and initially convex, while the full diagram $\mD=[n]^2$ is neither (unless $n=1$).}

\begin{definition}
    Let $p$ be a prime number. A Ferrers diagram $\mD=(c_1,\ldots,c_n)$ is called \textbf{strictly $p$-monotone} if its $p$-contraction is strictly monotone.
    A Ferrers diagram $\mD=(c_1,\ldots,c_n)$ is called \textbf{initially $p$-convex} if its $p$-contraction is initially convex.
\end{definition}

\begin{remark}\label{rmk:monotone-p}
Let $p$ be a prime number and let $\mD$ be
a strictly monotone Ferrers diagram. \rev{If} $\mD$ is nonempty, \rev{then} the $p$-height of $\mD$ is equal to $0$ and so $\mD^{(p)}=\mD$. In particular, strictly monotone Ferrers diagrams are strictly $p$-monotone for every choice of a prime $p$.
\end{remark}

The following result is straightforward.

\begin{lemma}\label{lem:adj-p-conv}
Let $p$ be a prime number.
    A Ferrers diagram $\mD$ is strictly $p$-monotone if and only if $\mD^\top$ is initially $p$-convex.
\end{lemma}

The following results ensures, up to iteration, that the property of being strictly $p$-monotone is preserved via top-right embeddings of Ferrers diagrams into larger square grids. 

\begin{lemma}\label{lem:s-monotone_embed}
Let $p$ be a prime number and let $\mD=(c_1,\ldots,c_n)$ be a strictly $p$-monotone Ferrers diagram of order $n$ of $p$-height $h$. Then 
$$\mD'=(\underbrace{\,0\ldots,0\,}_{p^h \textup{ times}},c_1,\ldots,c_n)$$
is a strictly $p$-monotone Ferrers diagram of order $n+p^h$ and $p$-height $h$.
\end{lemma}

Observe that Lemma \ref{lem:s-monotone_embed} is not valid if we only assume the Ferrers diagram to be monotone.

\begin{theorem}\label{th:strictly-p-monotone}
    Let $p$ be a prime number and let $d,n$ be positive integers with $1\le d \le n$. Then the following hold:
    \begin{enumerate}[label=(\arabic*)]
    \item Conjecture \ref{conj:ES} holds true for strictly $p$-monotone Ferrers diagrams \rev{of order $n$} over any finite field of characteristic $p$.
    \item Conjecture \ref{conj:ES} holds true for initially $p$-convex Ferrers diagrams \rev{of order $n$} over any finite field of characteristic $p$.
    \end{enumerate}
\end{theorem}
\begin{proof}
    We only prove the result for strictly $p$-monotone Ferrers diagrams: the one for initially $p$-convex Ferrers diagrams follows then from \cref{lem:adj-p-conv}. Let $\mD=(c_1,\ldots,c_n)$ be a strictly $p$-monotone Ferrers diagram of $p$-height $h$. Let $\F$ be a finite field of characteristic $p$ and $m=\min\{i \,:\, p^i \geq n\}$.
    Since $p^h$ divides $p^{m}-n$, 
    iterating \cref{lem:s-monotone_embed} precisely $(p^m-n)/p^h$ times, we embed $\mD$ into a strictly $p$-monotone Ferrers diagram $\mD'=(c_1',\ldots,c_{p^m}')$ of order $p^m$, where 
    $$c_i'=\begin{cases}
        0 & \mbox{ if } i\leq p^m-n, \\
        c_{i-\rev{(p^m-n)}} & \mbox { if } i>p^m-n.
    \end{cases}$$ 
    \cref{thm:MFD-monotone} ensures the existence of a $[\mD',\nu_{\min}(\mD',d),d]_\F$ MFD code, which is clearly isometric to a $[\mD,\nu_{\min}(\mD,d),d]_\F$ MFD code.
\end{proof}

We derive the following as a direct consequence of \cref{th:strictly-p-monotone} and \cref{rmk:monotone-p}.

\begin{theorem}\label{th:strictly-monotone}
    Let $d,n$ be positive integers with $1\le d \le n$. Then the following hold:
    \begin{enumerate}[label=(\arabic*)]
    \item Conjecture \ref{conj:ES} holds true for strictly monotone Ferrers diagrams \rev{of order $n$} over any finite field.
    \item Conjecture \ref{conj:ES} holds true for initially convex Ferrers diagrams \rev{of order $n$} over any finite field.
    \end{enumerate}
\end{theorem}

We remark that the last result holds over any finite field, in contrast with \cref{thm:MFD-monotone}, which requires the characteristic $p$ of the field
to be the unique prime divisor of $n$. 

\begin{corollary}\label{cor:upper-triangular}
    Conjecture \ref{conj:ES} holds true for upper triangular matrices over any finite field.
\end{corollary}

\begin{corollary}\label{cor:blocks}
If $p$ is a prime number and $\mD$ is a Ferrers diagram whose $p$-contraction is upper triangular, then \cref{conj:ES} holds true over any finite field of characteristic $p$.
\end{corollary}

We observe that \cref{cor:blocks} partially answers an open problem from \cite[Section~VIII]{etzion2016optimal} asking whether MFD codes exist for Ferrers diagrams of the form $(2,2,4,4,\ldots,2r,2r)$.

\begin{example}
   In this example we illustrate how to construct an MFD code of upper triangular matrices in the smallest \rev{formerly} open case over $\F_2$. For this, let $n=6$ and $d=4$ and consider the strictly monotone Ferrers diagram $\mD=(1,2,3,4,5,6)$. In this case, we have
   $\nu_{\min}(\mD,4)=6$, and now we show how to construct a $[\mD,6,4]_{\F_2}$ MFD code.

    The smallest power of $2$ which is at least $n$ is $n'=8$, and we can extend $\mD$ to the Ferrers diagram $\mD'=(0,0,1,2,3,4,5,6)$ of order $8$, which is still strictly monotone, as a consequence of Lemma \ref{lem:s-monotone_embed}. Thus, we consider the degree $n'=8$ extension field of $\F_2$ given by $\F_{2^8}=\F_2(\gamma)$, where $\gamma^8+\gamma^4+\gamma^3+\gamma^2+1=0$. As in Example \ref{exa:n=8_UT}, we let $\sigma$ be the $2$-Frobenius automorphism of $\F_{2^8}$, we set $\bsig=\sigma-\id$, and take the full flag $\mF=(\mF_0,\ldots,\mF_8)$, where $\mF_i=\ker(\bsig^i)$. The basis $\mB=(1,\gamma^{170},\gamma^{136},\gamma^{204},\gamma^{222},\gamma^{38},\gamma^{143},\gamma^{5})$ is $\mF$-compatible, and Theorem \ref{thm:monotone_representation} yields that $\F_2^\mD\cong\F_2^{\mD'}$ is isomorphic to 
    $$\F_{2^8}[\sigma;\mD']=\bigoplus_{i=3}^8 \mF_{i-2}\bsig^{i-1}=\left\{\sum_{i=3}^8 \lambda_i\bsig^{i-1} \,:\, \lambda_i \in \mF_{i-2}\right\},$$
    where the isomorphism is defined using the $\F_2$-basis $\mB$. 
    The $[\mD',6,4]_{\F_2}$ MFD code representation in $\F_{2^8}[\sigma;\mD']$ is then given by 
    $$\F_{2^8}[\sigma;\mD']_4=\langle \bsig^2,\bsig^3,\gamma^{170}\bsig^3,\bsig^4,\gamma^{170}\bsig^4,\gamma^{136}\bsig^4\rangle_{\F_2}.$$
    A straightforward computation shows that writing each element with respect to the $\F_2$-basis $\mB$ -- and then taking the top-right $6\times 6$ submatrices -- we obtain the $[\mD,6,4]_{\F_2}$ MFD code in $\F_2^{\mD}$ as the span of the following six matrices:
    $$\left(\begin{array}{cccccccc}   
     1  \cellcolor{blue!10}  & 0 \cellcolor{blue!10}  & 0 \cellcolor{blue!10}  & 0  \cellcolor{blue!10}  &  0  \cellcolor{blue!10}  &  0  \cellcolor{blue!10}  \\
    0  &  1  \cellcolor{blue!10} &  0  \cellcolor{blue!10} &  0  \cellcolor{blue!10} & 0  \cellcolor{blue!10} &  0  \cellcolor{blue!10} \\
    0 &  0 & 1  \cellcolor{blue!10}  & 0  \cellcolor{blue!10} &  0 \cellcolor{blue!10}  &  0 \cellcolor{blue!10}  \\
   0 &   0 &   0   & 1  \cellcolor{blue!10}  &   0  \cellcolor{blue!10}   &  0  \cellcolor{blue!10} \\
     0   &  0  &   0   &  0   &  1  \cellcolor{blue!10}  &   0  \cellcolor{blue!10} \\
      0   &  0  &   0   & 0   &  0 &   1  \cellcolor{blue!10} 
     \end{array}\right), \,
\left(\begin{array}{cccccccc}      
    0 \cellcolor{blue!10} &  1   \cellcolor{blue!10} &  0  \cellcolor{blue!10}  &  0  \cellcolor{blue!10}  &  0  \cellcolor{blue!10}  &  0  \cellcolor{blue!10} \\
   0  &  0  \cellcolor{blue!10}  &  1  \cellcolor{blue!10}  &  0  \cellcolor{blue!10}  &  0   \cellcolor{blue!10} &  0  \cellcolor{blue!10} \\
   0  &  0  &  \cellcolor{blue!10}  0  &  \cellcolor{blue!10}  1  &  \cellcolor{blue!10}  0  &  0 \cellcolor{blue!10}  \\
   0  &  0  &  0  &  0  \cellcolor{blue!10}  &  1  \cellcolor{blue!10}  &  0  \cellcolor{blue!10} \\
   0  &  0  &  0  &  0  &  0  \cellcolor{blue!10}  &  1 \cellcolor{blue!10}  \\
   0  &  0  &  0  &  0  &  0  &  0  \cellcolor{blue!10} 
   \end{array}\right),\,
\left(\begin{array}{cccccccc}   
   0 \cellcolor{blue!10} &  0  \cellcolor{blue!10}  &  1  \cellcolor{blue!10}  &  0   \cellcolor{blue!10} &  1  \cellcolor{blue!10}  &  0  \cellcolor{blue!10} \\
   0  &  1  \cellcolor{blue!10}  &  1  \cellcolor{blue!10}  &  1  \cellcolor{blue!10}  &  0  \cellcolor{blue!10}  &  1  \cellcolor{blue!10} \\
   0  &  0  &  0  \cellcolor{blue!10}  &  0  \cellcolor{blue!10}  &  1  \cellcolor{blue!10}  &  0  \cellcolor{blue!10} \\
   0  &  0  &  0  &  1  \cellcolor{blue!10}  &  1  \cellcolor{blue!10}  &  1  \cellcolor{blue!10} \\
   0  &  0  &  0  &  0  &  0  \cellcolor{blue!10}  &  0 \cellcolor{blue!10}  \\
   0  &  0  &  0  &  0  &  0  &  1 \cellcolor{blue!10}  
   \end{array}\right),   $$
   
   $$   \left(\begin{array}{cccccccc}       
   0  \cellcolor{blue!10}  &  0  \cellcolor{blue!10}  &  1  \cellcolor{blue!10}   &  0  \cellcolor{blue!10}  &  0  \cellcolor{blue!10}  &  0 \cellcolor{blue!10}  \\
     0  &  0  \cellcolor{blue!10}  &  0  \cellcolor{blue!10}  &  1  \cellcolor{blue!10}  &  0  \cellcolor{blue!10}  &  0  \cellcolor{blue!10} \\
     0  &  0  &  0 \cellcolor{blue!10}   &  0  \cellcolor{blue!10}  &  1  \cellcolor{blue!10}  &  0 \cellcolor{blue!10}  \\
    0  &  0  &  0  &  0  \cellcolor{blue!10}  &  0  \cellcolor{blue!10}  &  1 \cellcolor{blue!10}  \\
     0  &  0  &  0  &  0  &  0  \cellcolor{blue!10}  &  0 \cellcolor{blue!10}  \\
    0  &  0  &  0  &  0  &  0  &  0 \cellcolor{blue!10}  
   \end{array}\right),\,
   \left(\begin{array}{cccccccc}        
     0  \cellcolor{blue!10}  &  0  \cellcolor{blue!10}  &  0  \cellcolor{blue!10}  &  1  \cellcolor{blue!10}  &  0  \cellcolor{blue!10}  &  1 \cellcolor{blue!10}  \\
    0  &  0  \cellcolor{blue!10}  &  1  \cellcolor{blue!10}  &  1  \cellcolor{blue!10}  &  1  \cellcolor{blue!10}  &  0 \cellcolor{blue!10}  \\
     0  &  0  &  0  \cellcolor{blue!10}  &  0  \cellcolor{blue!10}  &  0  \cellcolor{blue!10}  &  1  \cellcolor{blue!10} \\
    0  &  0  &  0  &  0  \cellcolor{blue!10}  &  1  \cellcolor{blue!10}  &  1  \cellcolor{blue!10} \\
    0  &  0  &  0  &  0  &  0  \cellcolor{blue!10}  &  0  \cellcolor{blue!10} \\
     0  &  0  &  0  &  0  &  0  &  0 \cellcolor{blue!10}  
   \end{array}\right), \,
   \left(\begin{array}{cccccccc}         
     0  \cellcolor{blue!10}  &  0  \cellcolor{blue!10}  &  0  \cellcolor{blue!10}  &  0 \cellcolor{blue!10}   &  0 \cellcolor{blue!10}   &  1 \cellcolor{blue!10}  \\
    0  &  0  \cellcolor{blue!10}  &  0  \cellcolor{blue!10}  &  1  \cellcolor{blue!10}  &  1  \cellcolor{blue!10}  &  1 \cellcolor{blue!10}  \\
    0  &  0  &  1  \cellcolor{blue!10}  &  0  \cellcolor{blue!10}  &  1  \cellcolor{blue!10}  &  0  \cellcolor{blue!10} \\
     0  &  0  &  0  &  1  \cellcolor{blue!10}  &  0  \cellcolor{blue!10}  &  0 \cellcolor{blue!10}  \\
     0  &  0  &  0  &  0  &  0  \cellcolor{blue!10}  &  0  \cellcolor{blue!10} \\
    0  &  0  &  0  &  0  &  0  &  0 \cellcolor{blue!10}  
   \end{array}\right).  $$
   Observe that, if we wanted to construct a $[\mD,6,4]_{\F}$ MFD code over a finite field $\F$ of some other (positive) characteristic we would have to redefine $n'$. For instance, if we chose $\F=\F_3$, then  $\mD$ would be mapped to $\mD'=(0,0,0,1,2,3,4,5,6)$, which is of order $9=3^2$, and consequently the representation of $\F_3^{\mD'}$  would become  $\F_{3^9}[\sigma;\mD']$. 
\end{example}

\begin{remark}\label{rmk:Q}
    As a consequence of 
\cite[Corollary~1]{ballico2015linear}, if for a pair $(\mD,d)$ there exists a maximum Ferrers diagram code over some finite field $\F$, then such a code can be lifted to an MFD code over $\Q$. In particular, \cref{thm:MFD-monotone} and \cref{th:strictly-monotone} yield the existence of MFD codes over $\Q$ for $p$-monotone and strictly $p$-monotone Ferrers diagrams, and their adjoints. 
\end{remark}

\subsection{Extension to MDS-constructible pairs}

 We conclude the paper by extending Theorem \ref{th:strictly-monotone} to all MDS-constructible pairs $(\mD,d)$; cf.\ \cref{prop:smon_are_MDSconst}. We recall that, for MDS-constructible pairs \cref{conj:ES} is known to be true only over large finite fields; cf.\ \cref{sec:rank-metric}. Here, we will prove it for every finite field.
 We start by showing that strictly monotone Ferrers diagrams and their adjoints are MDS-constructible with respect to any sensible distance.

 \begin{proposition}\label{prop:smon_are_MDSconst}
     Let $\mD$ be a strictly monotone or initially convex Ferrers diagram of order $n$. Then, for every integer $2\le d \le n$, the pair $(\mD,d)$ is MDS-constructible. 
 \end{proposition}

\begin{proof}
 We start by assuming that $\mD$ is strictly monotone so, thanks to Remark \ref{rem:deletion-monotone}, we have that  $\nu_{\min}(\mD,d)=|\mD \cap ([n]\times[n-d+1])|$. Moreover, for every $i\in[n]$, strict monotonicity implies $c_i\le i$ and thus $\mD\subseteq \mT_n$.  
     With a slight abuse of notation identifying $\mT_{n-d+1}$ with its image in $[n]^2$, this means that   \begin{equation}\label{eq:numin_SM}\nu_{\min}(\mD,d)=|\mD\cap\mT_{n-d+1}|.
    \end{equation}
    Set $R=[n]\times([n]\setminus[n-d+1])$ and, for each $i \in [n-d+1]$, write $r_i=|\mD\cap \Delta_i^n\cap \mT_{n-d+1}|$. Assuming that $r_i>0$,  \cref{rem:monotone_diagonal} yields that
    $(\Delta_i^n\cap R)\subseteq \mD$ and thus
    \begin{equation}\label{eq:ri}r_i=|\mD\cap (\Delta_i^n\cap \mT_{n-d+1})|=|\mD\cap \Delta_i^n|-|\mD\cap (\Delta_i^n\cap R)|=|\mD\cap \Delta_i^n|-d+1.\end{equation}
    Combining \eqref{eq:numin_SM} and \eqref{eq:ri} we finally obtain
    $$\nu_{\min}(\mD,d)=|\mD\cap\mT_{n-d+1}|=\sum_{i=1}^{n-d+1}r_i=\sum_{i=1}^{n-d+1}\max\{0,|\mD\cap\Delta_i^n|-d+1\},$$
     which shows that the pair $(\mD,d)$ is MDS-constructible. This also concludes the proof as initially convex Ferrers diagrams are adjoints of strictly monotone ones and MDS-constructibility is preserved under adjunction.
\end{proof}

Observe that Proposition \ref{prop:smon_are_MDSconst} is not true in general if we consider the wider classes of monotone or $p$-monotone Ferrers diagrams (and their adjoints), as the next example shows. 

\begin{example}
In this example, for given positive integers $n$ and $2\leq d\leq n$, write 
$$\nu_{\mathrm{MDS}}(\mD,d)=\sum_{i=1}^{n-d+1}\max\{0,|\mD\cap\Delta_i^n|-d+1\}.$$
In the following table, we collect the values of $\nu_{\min}(\mD,d)$ and $\nu_{\mathrm{MDS}}(\mD,d)$ for the first two monotone Ferrers diagrams from \cref{exa:monotone}, which we label in order $\mD_1$ and $\mD_2$. Note that here $n=5$ and $\mD_1$ is strictly monotone, whilst $\mD_2$ is not. 
\begin{center}
    \begin{tabular}{c|c|c|c|c}
    $d$ & $\nu_{\min}(\mD_1,d)$ & $\nu_{\mathrm{MDS}}(\mD_1,d)$ & $\nu_{\min}(\mD_2,d)$ & $\nu_{\mathrm{MDS}}(\mD_2,d)$  \\
    \hline
       2  & 4 & 4 & {\bf 12} & {\bf 10}\\ \hline
       3  & 1 & 1 & {\bf 7}  & {\bf 6}\\ \hline
       4  & 0 & 0  & 3  & 3\\ \hline
       5  & 0 & 0 & 1  & 1
    \end{tabular}
\end{center}
\end{example}

\begin{definition}
 Let $\mD=(c_1,\ldots,c_n)$ be a Ferrers diagram of order $n$, let $d \in \{2,\ldots,n\}$ and let $j\in\{0,\ldots,d-1\}$. The pair  $(\mD,d)$ is called \textbf{$j$-Singleton} if 
 $$\nu_{\min}(\mD,d)\rev{=\nu_j(\mD,d)}=
 \sum_{i=1}^{n-j}\max\{0,c_i-d+1+j\}.$$
\end{definition}

Informally speaking, a pair $(\mD,d)$ is $j$-Singleton if $\nu_{\min}(\mD,d)$ is obtained by counting the number of dots in $\mD$ after removing the last $j$ columns and the first $d-j-1$ rows. Because of this and \eqref{eq:gen_SBbound}, every pair $(\mD,d)$, where $\mD$ is a Ferrers diagram of order $n$ and $2\le d \le n$, is $j$-Singleton for some $j\in\{0,\ldots,d-1\}$. It is not difficult to construct examples of pairs $(\mD,d)$ that are $j$-Singleton for more than one choice of $j$. 

\medskip

Until the end of the paper, we will use the following notation. For any positive integers $n$ and $d$ with $2\le d \le n$ and every $j \in \{0,\ldots,d-1\}$, define 
%the quantity 
%\begin{align*}
%\nu_j(\mD,d) & = \sum_{i=1}^{n-j}\max\{0,c_i-d+1+j\}, 
%\end{align*}
%and also 
the following subsets of $[n]^2$:
\begin{align*}
\mS_{n,d,j}& =\{(i,\ell) \,:\, i \in \{d-j,\ldots, n\}, \ell  \in [n-j] \}, \\
\mT_{n,d,j}&=\mS_{n,d,j}\cap \mT_{n}, \\
 \mL_{n,d,j}&=[n]^2\setminus \mS_{n,d,j}.
\end{align*}
 With the notation above, observe that, for any Ferrers diagram $\mD$ of order $n$, one has
$$\sum_{i=1}^{n-j}\max\{0,c_i-d+1+j\}=| \mD\cap \mS_{n,d,j}|=|\mD|-|\mD\cap \mL_{n,d,j}|=\nu_j(\mD,d),$$
and hence, the value $\nu_{\min}(\mD,d)$ in \eqref{eq:gen_SBbound} can be written as
$$ \rev{\nu_{\min}(\mD,d)=\min\left\{\nu_j(\mD,d)\ :\ j\in\{0,\ldots, d-1\}\right\}=\min\left\{ | \mD\cap \mS_{n,d,j}|\ :\ j\in\{0,\ldots, d-1\}\right\}. }$$
In particular, $(\mD,d)$ is $j$-Singleton if and only if $\nu_{\min}(\mD,d)=\nu_j(\mD,d)$.

\begin{example}\label{ex:STL}
    For $n=8$, $d=4$ and $j=1$, the figure below represents the three sets $\mS_{8,4,1}$, $\mT_{8,4,1}$ and $\mL_{8,4,1}$ as follows. The black dots represent $\mS_{8,4,1}$, the orange dots represent $\mL_{8,4,1}$ and the dots contained in the red triangular area correspond to $\mT_{8,4,1}$. 
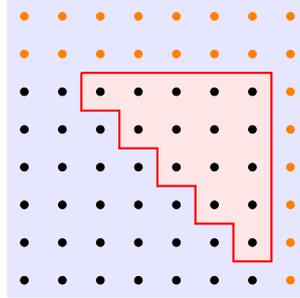
\begin{center}
    \begin{tikzpicture}[scale=0.5]
\draw[help lines, very thick, white, fill=blue!10] (0,1) -- (0,-7) -- (8,-7)--(8,1)--(0,1);
\draw[help lines,  thick, red, fill=red!10] (2,-1) -- (7,-1) -- (7,-6)--(6,-6)--(6,-5)--(5,-5)--(5,-4)--(4,-4)--(4,-3)--(3,-3)--(3,-2)--(2,-2)--(2,-1);
\draw[orange,fill=orange] (0.5,0.5) circle (0.1cm);
\draw[orange,fill=orange]  (1.5,0.5) circle (0.1cm);
\draw[orange,fill=orange]  (2.5,0.5) circle (0.1cm);
\draw[orange,fill=orange] (3.5,0.5) circle (0.1cm);
\draw[orange,fill=orange] (4.5,0.5) circle (0.1cm);
\draw[orange,fill=orange] (5.5,0.5) circle (0.1cm);
\draw[orange,fill=orange] (6.5,0.5) circle (0.1cm);
\draw[orange,fill=orange] (7.5,0.5) circle (0.1cm);

\draw[orange,fill=orange] (0.5,-0.5) circle (0.1cm);
\draw[orange,fill=orange] (1.5,-0.5) circle (0.1cm);
\draw[orange,fill=orange] (2.5,-0.5) circle (0.1cm);
\draw[orange,fill=orange] (3.5,-0.5) circle (0.1cm);
\draw[orange,fill=orange] (4.5,-0.5) circle (0.1cm);
\draw[orange,fill=orange] (5.5,-0.5) circle (0.1cm);
\draw[orange,fill=orange] (6.5,-0.5) circle (0.1cm);
\draw[orange,fill=orange] (7.5,-0.5) circle (0.1cm);

\draw[fill=black] (0.5,-1.5) circle (0.1cm);
\draw[fill=black] (1.5,-1.5) circle (0.1cm);
\draw[fill=black] (2.5,-1.5) circle (0.1cm);
\draw[fill=black] (3.5,-1.5) circle (0.1cm);
\draw[fill=black] (4.5,-1.5) circle (0.1cm);
\draw[fill=black] (5.5,-1.5) circle (0.1cm);
\draw[fill=black] (6.5,-1.5) circle (0.1cm);
\draw[orange,fill=orange] (7.5,-1.5) circle (0.1cm);

\draw[fill=black] (0.5,-2.5) circle (0.1cm);
\draw[fill=black] (1.5,-2.5) circle (0.1cm);
\draw[fill=black] (2.5,-2.5) circle (0.1cm);
\draw[fill=black] (3.5,-2.5) circle (0.1cm);
\draw[fill=black] (4.5,-2.5) circle (0.1cm);
\draw[fill=black] (5.5,-2.5) circle (0.1cm);
\draw[fill=black] (6.5,-2.5) circle (0.1cm);
\draw[orange,fill=orange] (7.5,-2.5) circle (0.1cm);

\draw[fill=black] (0.5,-3.5) circle (0.1cm);
\draw[fill=black] (1.5,-3.5) circle (0.1cm);
\draw[fill=black] (2.5,-3.5) circle (0.1cm);
\draw[fill=black] (3.5,-3.5) circle (0.1cm);
\draw[fill=black] (4.5,-3.5) circle (0.1cm);
\draw[fill=black] (5.5,-3.5) circle (0.1cm);
\draw[fill=black] (6.5,-3.5) circle (0.1cm);
\draw[orange,fill=orange] (7.5,-3.5) circle (0.1cm);

\draw[fill=black] (0.5,-4.5) circle (0.1cm);
\draw[fill=black] (1.5,-4.5) circle (0.1cm);
\draw[fill=black] (2.5,-4.5) circle (0.1cm);
\draw[fill=black] (3.5,-4.5) circle (0.1cm);
\draw[fill=black] (4.5,-4.5) circle (0.1cm);
\draw[fill=black] (5.5,-4.5) circle (0.1cm);
\draw[fill=black] (6.5,-4.5) circle (0.1cm);
\draw[orange,fill=orange] (7.5,-4.5) circle (0.1cm);

\draw[fill=black] (0.5,-5.5) circle (0.1cm);
\draw[fill=black] (1.5,-5.5) circle (0.1cm);
\draw[fill=black] (2.5,-5.5) circle (0.1cm);
\draw[fill=black] (3.5,-5.5) circle (0.1cm);
\draw[fill=black] (4.5,-5.5) circle (0.1cm);
\draw[fill=black] (5.5,-5.5) circle (0.1cm);
\draw[fill=black] (6.5,-5.5) circle (0.1cm);
\draw[orange,fill=orange] (7.5,-5.5) circle (0.1cm);

\draw[fill=black] (0.5,-6.5) circle (0.1cm);
\draw[fill=black] (1.5,-6.5) circle (0.1cm);
\draw[fill=black] (2.5,-6.5) circle (0.1cm);
\draw[fill=black] (3.5,-6.5) circle (0.1cm);
\draw[fill=black] (4.5,-6.5) circle (0.1cm);
\draw[fill=black] (5.5,-6.5) circle (0.1cm);
\draw[fill=black] (6.5,-6.5) circle (0.1cm);
\draw[orange,fill=orange] (7.5,-6.5) circle (0.1cm);
\end{tikzpicture}
\end{center}
\end{example}

\begin{lemma}\label{lem:jSing}
 Let $\mD$ be a Ferrers diagram of order $n$, let $d \in \{2,\ldots,n\}$, and let $j\in\{0,\ldots,d-1\}$ be such that $(\mD,d)$ is $j$-Singleton. Let $\mD'\subseteq\mD$ be a Ferrers diagram of order $n$ with the property that $\mD\cap \mL_{n,d,j}=\mD'\cap \mL_{n,d,j}$. Then $(\mD',d)$ is $j$-Singleton. 
 \end{lemma}

\begin{proof}
For any $i \in \{0,\ldots,d-1\}$ we have
\begin{align}\label{eq:nu_i}\nu_{\min}(\mD,d)\leq \nu_i(\mD,d)&=|\mD|-|\mD\cap \mL_{n,d,i}|=|\mD'|+|\mD\setminus \mD'|-|\mD\cap \mL_{n,d,i}| \nonumber \\
 &\leq |\mD'|+|\mD\setminus \mD'|-|\mD'\cap \mL_{n,d,i}|=\nu_i(\mD',d)+|\mD\setminus \mD'|.
 \end{align}
 On the other hand, by hypothesis $(\mD,d)$ is $j$-Singleton, and thus we have
 \begin{equation}\label{eq:nu_j}\nu_{\min}(\mD,d)=\nu_j(\mD,d)=|\mD\cap \mS_{n,d,j}|=|\mD|-|\mD\cap \mL_{n,d,j}|=\nu_j(\mD',d)+|\mD\setminus \mD'|.
 \end{equation}
 Combining \eqref{eq:nu_i} with \eqref{eq:nu_j}, we get $\nu_{\min}(\mD',d)=\nu_j(\mD',d)$.
\end{proof}

\begin{lemma}\label{lem:MDScolumn}
    Let $\mD$ be a Ferrers diagram of order $n$ and let $d \in \{2,\ldots,n\}$, $j\in\{0,\ldots,d-1\}$. Assume that $(\mD,d)$ is MDS-constructible and $j$-Singleton. Then the following hold:
    \begin{enumerate}[label=(\arabic*)]
    \item\label{lem:MDS1} One has $\mD \cap \mS_{n,d,j}=\mD \cap \mT_{n,d,j}$. 
    \item\label{lem:MDS1.5} One has $\{i \in[n-d+1] \,:\, |\mD\cap\Delta_i^n|\geq d\}=\{i \in[n-d+1]\,: \, \mD\cap\Delta_i^n\cap \mS_{n,d,j} \neq \emptyset \}$.
    \item\label{lem:MDS2} If $i \in [n-d+1]$ and $\mD\cap \Delta_i^n\cap \mS_{n,d,j}\neq \emptyset$, then 
    $\mD\cap \Delta_i^n\supseteq \Delta_i^n\cap \mL_{n,d,j}$.
    \end{enumerate}
\end{lemma}

\begin{proof}
     Define the subsets $X$ and $Y$ of $[n-d+1]$ as
     $$
     X=\{i \in[n-d+1] \,:\, |\mD\cap\Delta_i^n|\geq d\} \ \ \textup{ and } \ \  Y=\{i \in[n-d+1]\,: \, \mD\cap\Delta_i^n\cap \mS_{n,d,j} \neq \emptyset \}.
     $$
     \rev{We claim that $X\subseteq Y$. To see this, let $i\in X$ and note that  $|\Delta_i^n\cap \mathcal{S}_{n,d,j}|=n-d-i+2$. Since $|\mD\cap\Delta_i^n|\geq d$ and both $\Delta_i^n\cap \mathcal{S}_{n,d,j}$ and $\mD\cap\Delta_i^n$ are contained in the set $\Delta_i^n$ of cardinality $n-i+1$, they must intersect nontrivially. Since $i$ was chosen arbitrarily, the claim is proven.
     
     We proceed by first proving \ref{lem:MDS1} and then  \ref{lem:MDS1.5} and \ref{lem:MDS2} together.}
     
\ref{lem:MDS1}  Since $(\mD,d)$ is $j$-Singleton, we have that
\begin{equation}\label{eq:11}\nu_{\min}(\mD,d)=|\mD\cap \mS_{n,d,j}|.
    \end{equation}
    On the other hand, $(\mD,d)$ is MDS-constructible and thus we have 
\begin{align}\label{eq:numin_ineq}\nu_{\min}(\mD,d)&=\sum_{i\in X}(|\mD\cap\Delta_i^n|-d+1)=\sum_{i\in X}(|\mD\cap\Delta_i^n|-|\Delta_i^n\cap \mL_{n,d,j}|) \nonumber\\
    &\leq \sum_{i\in X}(|\mD\cap\Delta_i^n|-|\mD\cap(\Delta_i^n\cap \mL_{n,d,j})|)=\sum_{i\in X}(|\mD \cap(\Delta_i^n\cap \mS_{n,d,j})|) \\
    &\le \sum_{i\in Y}(|\mD \cap(\Delta_i^n\cap \mS_{n,d,j})|) =|\mD \cap \mT_{n,d,j}|.\nonumber
    \end{align}
    Combining \eqref{eq:11} and \eqref{eq:numin_ineq}, we derive that $\mD\cap \mS_{n,d,j}=\mD\cap \mT_{n,d,j}$.

\ref{lem:MDS1.5}-\ref{lem:MDS2} Thanks to the proof of \ref{lem:MDS1}, we know that the inequalities in \eqref{eq:numin_ineq} are in fact  all equalities and so, for each $i\in X$ we have $\mD\cap (\Delta_i^n\cap \mL_{n,d,j})=\Delta_i^n\cap \mL_{n,d,j}$. Furthermore, by the definition of $Y$ and using the fact that the second to last inequality in \eqref{eq:numin_ineq} is an equality, we deduce that $Y=X$ and the proof is complete.  
\end{proof}

\begin{example}\label{exa:MDS-j}
Let $\mD=(0,2,2,3,3,5,6,8)$ be a Ferrers diagram of order $8$ and let $d=4$. It is easy to verify that $(\mD,4)$ is MDS-constructible and $j$-Singleton for $j=1$. We can see that $\mD\cap \mS_{8,4,1}=\mD\cap\mT_{8,4,1}$. Moreover, for every index $2\le i \le 5$ we have that $\mD\cap \Delta_i^8\cap \mT_{8,4,1}\neq 0$   and indeed, for each such index, it holds that $\mD\cap \Delta_i^8\supseteq \Delta_i^8\cap \mL_{8,4,1}$.
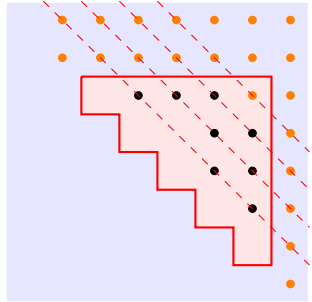
\begin{center}
       \begin{tikzpicture}[scale=0.5]
\draw[help lines, very thick, white, fill=blue!10] (0,1) -- (0,-7) -- (8,-7)--(8,1)--(0,1);
\draw[help lines,  thick, red, fill=red!10] (2,-1) -- (7,-1) -- (7,-6)--(6,-6)--(6,-5)--(5,-5)--(5,-4)--(4,-4)--(4,-3)--(3,-3)--(3,-2)--(2,-2)--(2,-1);
\draw[orange,fill=orange]  (1.5,0.5) circle (0.1cm);
\draw[orange,fill=orange]  (2.5,0.5) circle (0.1cm);
\draw[orange,fill=orange] (3.5,0.5) circle (0.1cm);
\draw[orange,fill=orange] (4.5,0.5) circle (0.1cm);
\draw[orange,fill=orange] (5.5,0.5) circle (0.1cm);
\draw[orange,fill=orange] (6.5,0.5) circle (0.1cm);
\draw[orange,fill=orange] (7.5,0.5) circle (0.1cm);

\draw[orange,fill=orange] (1.5,-0.5) circle (0.1cm);
\draw[orange,fill=orange] (2.5,-0.5) circle (0.1cm);
\draw[orange,fill=orange] (3.5,-0.5) circle (0.1cm);
\draw[orange,fill=orange] (4.5,-0.5) circle (0.1cm);
\draw[orange,fill=orange] (5.5,-0.5) circle (0.1cm);
\draw[orange,fill=orange] (6.5,-0.5) circle (0.1cm);
\draw[orange,fill=orange] (7.5,-0.5) circle (0.1cm);

\draw[fill=black] (3.5,-1.5) circle (0.1cm);
\draw[fill=black] (4.5,-1.5) circle (0.1cm);
\draw[fill=black] (5.5,-1.5) circle (0.1cm);
\draw[fill=black] (6.5,-1.5) circle (0.1cm);
\draw[orange,fill=orange] (7.5,-1.5) circle (0.1cm);

\draw[fill=black] (5.5,-2.5) circle (0.1cm);
\draw[fill=black] (6.5,-2.5) circle (0.1cm);
\draw[orange,fill=orange] (7.5,-2.5) circle (0.1cm);

\draw[fill=black] (5.5,-3.5) circle (0.1cm);
\draw[fill=black] (6.5,-3.5) circle (0.1cm);
\draw[orange,fill=orange] (7.5,-3.5) circle (0.1cm);

\draw[fill=black] (6.5,-4.5) circle (0.1cm);
\draw[orange,fill=orange] (7.5,-4.5) circle (0.1cm);

\draw[orange,fill=orange] (7.5,-5.5) circle (0.1cm);

\draw[orange,fill=orange] (7.5,-6.5) circle (0.1cm);
\draw[help lines,  dashed, red, thin] (1,1) -- (8,-6);
\draw[help lines,  dashed, red, thin] (2,1) -- (8,-5);
\draw[help lines,  dashed, red, thin] (3,1) -- (8,-4);
\draw[help lines,  dashed, red, thin] (4,1) -- (8,-3);
\end{tikzpicture}
\end{center}
\end{example}

The following is the main result of this section, which ensures in particular that Conjecture \ref{conj:ES} holds true for every MDS-constructible pair $(\mD,d)$.

\begin{theorem}\label{thm:MDS-constr-column}
    Let $\mD$ be a Ferrers diagram of order $n$ and let $2\leq d \leq n$ be an integer. If $(\mD,d)$ is MDS-constructible, then there exists a  $[\mD,\nu_{\min}(\mD,d),d]_\F$ MFD code over any finite field $\F$.
\end{theorem}

Before giving a proof of Theorem \ref{thm:MDS-constr-column}, we recall a useful criterion to construct MFD codes on a given Ferrers diagram from MFD codes on  related Ferrers diagrams. This strategy was first adopted in the proof of \cite[Theorem 7]{etzion2016optimal} and the explicit statement can be found in \cite[Remarks II.12 \& II.14]{antrobus2019maximal}.

\begin{proposition}\label{prop:MFD-from-old}
    Let $\mD$ be a Ferrers diagram of order $n$ and let $2\le d \le n$ be an integer.
    \begin{enumerate}[label=(\arabic*)]
        \item\label{it:old1} Let $\mD'\subseteq \mD$ be a Ferrers diagram of order $n$ such that $\nu_{\min}(\mD,d)=\nu_{\min}(\mD',d)$. If $\C$ is a $[\mD',\nu_{\min}(\mD',d),d]_\F$ MFD code, then it is also a $[\mD,\nu_{\min}(\mD,d),d]_\F$ MFD code.
        \item\label{it:old2} Let $\mD'\supseteq \mD$ be a Ferrers diagram of order $n$ such that $\nu_{\min}(\mD,d)=\nu_{\min}(\mD',d)-|\mD'\setminus\mD|$. If $\C$ is a $[\mD',\nu_{\min}(\mD',d),d]_\F$ MFD code, then $\C\cap \F^{\mD}$ is a $[\mD,\nu_{\min}(\mD,d),d]_\F$ MFD code.
    \end{enumerate}
\end{proposition}

We give here the proof of \cref{thm:MDS-constr-column} and we illustrate its steps explicitly in \cref{ex:Th17}.

\begin{proof}[Proof of \cref{thm:MDS-constr-column}]
  Let $\F$ be a finite field, let $(\mD,d)$ be an MDS-constructible pair and let $j \in \{0,\ldots,d-1\}$ be such that $(\mD,d)$ is $j$-Singleton.
 Define
     $$ 
     Y=\{i\in [n-d+1] \,: \, \mD\cap\Delta_i^n\cap \mS_{n,d,j} \neq \emptyset \}
     $$
     and note that, as a consequence of \eqref{eq:MDS-smallsum} and \cref{lem:MDScolumn}\ref{lem:MDS1.5}, one has 
     \begin{align*}
     \nu_{\min}(\mD,d)=\sum_{i=1}^{n}\max\{0,|\mD\cap\Delta_i^n|-d+1\}& =\sum_{i=1}^{n-d+1}\max\{0,|\mD\cap\Delta_i^n|-d+1\}\\ &=\sum_{i\in Y}\max\{0,|\mD\cap\Delta_i^n|-d+1\}.
     \end{align*}
     If $Y=\emptyset$, then $\nu_{\min}(\mD,d)=0$ and there is nothing to prove. Hence, assume $Y\neq\emptyset$. By the definition of $Y$, there exists $\ell \in [n-d+1]$ such that $Y=\{ i \,:\, \ell \le i \le n-d+1\}$. Fix such $\ell$ and define \rev{the following new Ferrers diagram}
     $$\mD'=(\mD\cap \mS_{n,d,j}) \cup \Big(\bigcup_{i=\ell}^n(\Delta_i^n\cap \mL_{n,d,j})\Big).$$
     By \cref{lem:MDScolumn}\ref{lem:MDS2}, for every $\ell \le i\le n-d+1$  we have that $\Delta_i^n\cap \mL_{n,d,j}\subseteq \mD$ while the same is satisfied for $i\geq n-d+2$ thanks to the fact that \rev{$\mD$ is a Ferrers diagram}. It follows in particular that $\mD'\subseteq \mD$. Moreover, it also holds that  $$\mD'\subseteq \mD''=\bigcup_{i=\ell}^n\Delta_i^n.$$
     The Ferrers diagram $\mD''$ is a copy of $\mT_{n-\ell+1}$ (considered as embedded into $[n]^2$) and hence it is strictly monotone and  $(\mD'',d)$ is $j'$-Singleton for every $j'$. In particular, it is $j$-Singleton. Furthermore, $\mD'$ and $\mD''$ have the same intersection with $\mL_{n,d,j}$. 
     Thus, Lemma \ref{lem:jSing} yields that $(\mD',d)$ is also $j$-Singleton.
     Since $\mD''$ is strictly monotone, by Theorem \ref{th:strictly-monotone} there exists a $[\mD'',\nu_{\min}(\mD'',d),d]_\F$ MFD code $\C''$, which we fix. Since $(\mD',d)$ and $(\mD'',d)$ are both $j$-Singleton and $\mD'\cap \mL_{n,d,j}=\mD''\cap \mL_{n,d,j}$, we compute
     \begin{align*}\nu_{\min}(\mD'',d)&=|\mD''\cap \mS_{n,d,j}|=|(\mD'\cap \mS_{n,d,j})|+|(\mD''\setminus\mD')\cap \mS_{n,d,j}|\\ &=\nu_{\min}(\mD',d)+|(\mD''\setminus\mD')\cap \mS_{n,d,j}|=\nu_{\min}(\mD',d)+|\mD''\setminus\mD'|.\end{align*}
    \cref{prop:MFD-from-old}\ref{it:old2} yields that $\C'=\C''\cap \F^{\mD'}$ is a $[\mD',\nu_{\min}(\mD',d),d]_\F$ MFD code. Finally, we have
    $$\nu_{\min}(\mD',d)=|\mD'\cap \mS_{n,d,j}|=|\mD\cap \mS_{n,d,j}|=\nu_{\min}(\mD,d)$$
    and $\mD'\subseteq \mD$. By Proposition \ref{prop:MFD-from-old}\ref{it:old1}, the code $\C'$ is thus also a $[\mD,\nu_{\min}(\mD,d),d]_{\F}$ MFD code.
     \end{proof}

  \begin{example}\label{ex:Th17}
      Let $\mD=(0,2,2,3,3,5,6,8)$ be the Ferrers Diagram from \cref{exa:MDS-j} and let $\F$ be an arbitrary finite field. In this case $(\mD,4)$ is $1$-Singleton and $\nu_{\min}(\mD,4)=\nu_1(\mD,4)=9$. In order to construct a $[\mD,9,4]_{\F}$ MFD code, we follow the proof of Theorem \ref{thm:MDS-constr-column} step by step. With the terminology introduced there and as illustrated below, we have
      $$Y=\{i\in[5] \,: \, \mD\cap\Delta_i^n\cap \mS_{n,d,j} \neq \emptyset \}=\{2,3,4,5\},$$
      and 
      $$\mD'=(\mD\cap \mS_{8,4,1}) \cup \Big(\bigcup_{i=2}^n(\Delta_i^n\cap \mL_{8,4,1})\Big)=(0,1,2,3,3,5,6,7).$$
      Furthermore,
      we have $\mD''=(0,1,2,3,4,5,6,7)$, as given in the third picture below. 
      
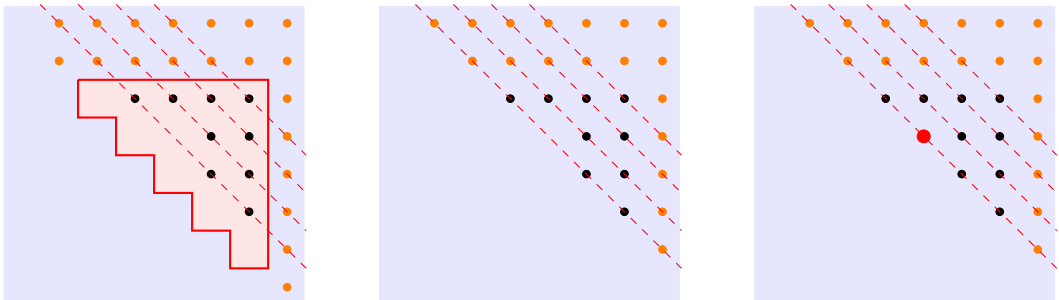
\begin{center}
       \begin{tikzpicture}[scale=0.5]
\draw[help lines, very thick, white, fill=blue!10] (0,1) -- (0,-7) -- (8,-7)--(8,1)--(0,1);
\draw[help lines,  thick, red, fill=red!10] (2,-1) -- (7,-1) -- (7,-6)--(6,-6)--(6,-5)--(5,-5)--(5,-4)--(4,-4)--(4,-3)--(3,-3)--(3,-2)--(2,-2)--(2,-1);

\draw[orange,fill=orange]  (1.5,0.5) circle (0.1cm);
\draw[orange,fill=orange]  (2.5,0.5) circle (0.1cm);
\draw[orange,fill=orange] (3.5,0.5) circle (0.1cm);
\draw[orange,fill=orange] (4.5,0.5) circle (0.1cm);
\draw[orange,fill=orange] (5.5,0.5) circle (0.1cm);
\draw[orange,fill=orange] (6.5,0.5) circle (0.1cm);
\draw[orange,fill=orange] (7.5,0.5) circle (0.1cm);

\draw[orange,fill=orange] (1.5,-0.5) circle (0.1cm);
\draw[orange,fill=orange] (2.5,-0.5) circle (0.1cm);
\draw[orange,fill=orange] (3.5,-0.5) circle (0.1cm);
\draw[orange,fill=orange] (4.5,-0.5) circle (0.1cm);
\draw[orange,fill=orange] (5.5,-0.5) circle (0.1cm);
\draw[orange,fill=orange] (6.5,-0.5) circle (0.1cm);
\draw[orange,fill=orange] (7.5,-0.5) circle (0.1cm);

\draw[fill=black] (3.5,-1.5) circle (0.1cm);
\draw[fill=black] (4.5,-1.5) circle (0.1cm);
\draw[fill=black] (5.5,-1.5) circle (0.1cm);
\draw[fill=black] (6.5,-1.5) circle (0.1cm);
\draw[orange,fill=orange] (7.5,-1.5) circle (0.1cm);

\draw[fill=black] (5.5,-2.5) circle (0.1cm);
\draw[fill=black] (6.5,-2.5) circle (0.1cm);
\draw[orange,fill=orange] (7.5,-2.5) circle (0.1cm);

\draw[fill=black] (5.5,-3.5) circle (0.1cm);
\draw[fill=black] (6.5,-3.5) circle (0.1cm);
\draw[orange,fill=orange] (7.5,-3.5) circle (0.1cm);

\draw[fill=black] (6.5,-4.5) circle (0.1cm);
\draw[orange,fill=orange] (7.5,-4.5) circle (0.1cm);

\draw[orange,fill=orange] (7.5,-5.5) circle (0.1cm);

\draw[orange,fill=orange] (7.5,-6.5) circle (0.1cm);
\draw[help lines,  dashed, red, thin] (1,1) -- (8,-6);
\draw[help lines,  dashed, red, thin] (2,1) -- (8,-5);
\draw[help lines,  dashed, red, thin] (3,1) -- (8,-4);
\draw[help lines,  dashed, red, thin] (4,1) -- (8,-3);
\end{tikzpicture} \qquad 
       \begin{tikzpicture}[scale=0.5]
\draw[help lines, very thick, white, fill=blue!10] (0,1) -- (0,-7) -- (8,-7)--(8,1)--(0,1);

\draw[orange,fill=orange]  (1.5,0.5) circle (0.1cm);
\draw[orange,fill=orange]  (2.5,0.5) circle (0.1cm);
\draw[orange,fill=orange] (3.5,0.5) circle (0.1cm);
\draw[orange,fill=orange] (4.5,0.5) circle (0.1cm);
\draw[orange,fill=orange] (5.5,0.5) circle (0.1cm);
\draw[orange,fill=orange] (6.5,0.5) circle (0.1cm);
\draw[orange,fill=orange] (7.5,0.5) circle (0.1cm);

\draw[orange,fill=orange] (2.5,-0.5) circle (0.1cm);
\draw[orange,fill=orange] (3.5,-0.5) circle (0.1cm);
\draw[orange,fill=orange] (4.5,-0.5) circle (0.1cm);
\draw[orange,fill=orange] (5.5,-0.5) circle (0.1cm);
\draw[orange,fill=orange] (6.5,-0.5) circle (0.1cm);
\draw[orange,fill=orange] (7.5,-0.5) circle (0.1cm);

\draw[fill=black] (3.5,-1.5) circle (0.1cm);
\draw[fill=black] (4.5,-1.5) circle (0.1cm);
\draw[fill=black] (5.5,-1.5) circle (0.1cm);
\draw[fill=black] (6.5,-1.5) circle (0.1cm);
\draw[orange,fill=orange] (7.5,-1.5) circle (0.1cm);

\draw[fill=black] (5.5,-2.5) circle (0.1cm);
\draw[fill=black] (6.5,-2.5) circle (0.1cm);
\draw[orange,fill=orange] (7.5,-2.5) circle (0.1cm);

\draw[fill=black] (5.5,-3.5) circle (0.1cm);
\draw[fill=black] (6.5,-3.5) circle (0.1cm);
\draw[orange,fill=orange] (7.5,-3.5) circle (0.1cm);

\draw[fill=black] (6.5,-4.5) circle (0.1cm);
\draw[orange,fill=orange] (7.5,-4.5) circle (0.1cm);

\draw[orange,fill=orange] (7.5,-5.5) circle (0.1cm);

\draw[help lines,  dashed, red, thin] (1,1) -- (8,-6);
\draw[help lines,  dashed, red, thin] (2,1) -- (8,-5);
\draw[help lines,  dashed, red, thin] (3,1) -- (8,-4);
\draw[help lines,  dashed, red, thin] (4,1) -- (8,-3);
\end{tikzpicture} \qquad 
       \begin{tikzpicture}[scale=0.5]
\draw[help lines, very thick, white, fill=blue!10] (0,1) -- (0,-7) -- (8,-7)--(8,1)--(0,1);
\draw[orange,fill=orange]  (1.5,0.5) circle (0.1cm);
\draw[orange,fill=orange]  (2.5,0.5) circle (0.1cm);
\draw[orange,fill=orange] (3.5,0.5) circle (0.1cm);
\draw[orange,fill=orange] (4.5,0.5) circle (0.1cm);
\draw[orange,fill=orange] (5.5,0.5) circle (0.1cm);
\draw[orange,fill=orange] (6.5,0.5) circle (0.1cm);
\draw[orange,fill=orange] (7.5,0.5) circle (0.1cm);

\draw[orange,fill=orange] (2.5,-0.5) circle (0.1cm);
\draw[orange,fill=orange] (3.5,-0.5) circle (0.1cm);
\draw[orange,fill=orange] (4.5,-0.5) circle (0.1cm);
\draw[orange,fill=orange] (5.5,-0.5) circle (0.1cm);
\draw[orange,fill=orange] (6.5,-0.5) circle (0.1cm);
\draw[orange,fill=orange] (7.5,-0.5) circle (0.1cm);

\draw[fill=black] (3.5,-1.5) circle (0.1cm);
\draw[fill=black] (4.5,-1.5) circle (0.1cm);
\draw[fill=black] (5.5,-1.5) circle (0.1cm);
\draw[fill=black] (6.5,-1.5) circle (0.1cm);
\draw[orange,fill=orange] (7.5,-1.5) circle (0.1cm);

\draw[red, fill=red] (4.5,-2.5) circle (0.17cm);
\draw[fill=black] (5.5,-2.5) circle (0.1cm);
\draw[fill=black] (6.5,-2.5) circle (0.1cm);
\draw[orange,fill=orange] (7.5,-2.5) circle (0.1cm);

\draw[fill=black] (5.5,-3.5) circle (0.1cm);
\draw[fill=black] (6.5,-3.5) circle (0.1cm);
\draw[orange,fill=orange] (7.5,-3.5) circle (0.1cm);

\draw[fill=black] (6.5,-4.5) circle (0.1cm);
\draw[orange,fill=orange] (7.5,-4.5) circle (0.1cm);

\draw[orange,fill=orange] (7.5,-5.5) circle (0.1cm);

\draw[help lines,  dashed, red, thin] (1,1) -- (8,-6);
\draw[help lines,  dashed, red, thin] (2,1) -- (8,-5);
\draw[help lines,  dashed, red, thin] (3,1) -- (8,-4);
\draw[help lines,  dashed, red, thin] (4,1) -- (8,-3);
\end{tikzpicture}
      \end{center}
      It is easy to verify that $\nu_{\min}(\mD',4)=\nu_{\min}(\mD,4)=9$ and $\nu_{\min}(\mD'',4)=10=\nu_{\min}(\mD',4)+|\mD''\setminus\mD'|$. 
    Theorem \ref{th:strictly-monotone} ensures the existence of a $[\mD'',10,4]_{\F}$ MFD code $\C''$. Then, by Proposition  \ref{prop:MFD-from-old}\ref{it:old2}, the code $\C'=\C''\cap \F^{\mD'}$  is a $[\mD',9,4]_{\F}$ MFD code and it is also a $[\mD,9,4]_{\F}$ MFD code thanks to  Proposition \ref{prop:MFD-from-old}\ref{it:old1}.
  \end{example}

\end{document}